\documentclass[a4paper,10pt,leqno]{amsart}
        \title{Survey on the Farrell-Jones Conjecture}
        \author{L\"uck, Wolfgang}
        \address{Mathematicians Institut der Universit\"at Bonn\\
                Endenicher Allee 60\\
                53115 Bonn, Germany}
         \email{wolfgang.lueck@him.uni-bonn.de}
          \urladdr{http://www.him.uni-bonn.de/lueck}
         \date{July, 2025}
     \keywords{algebraic K- and L-theory of group rings, Farrell-Jones Conjecture}
     \subjclass[2010]{primary: 19-02 secondary: 19A31, 19B28, 19G24 }

\usepackage{hyperref}
\usepackage{color}
\usepackage{pdfsync}
\usepackage{calc}
\usepackage{enumerate,amssymb}
\usepackage{graphicx}
\usepackage{datetime2}
\usepackage[arrow,curve,matrix,tips,2cell]{xy}

\DeclareMathAlphabet\EuR{U}{eur}{m}{n}
\SetMathAlphabet\EuR{bold}{U}{eur}{b}{n}

\newcommand{\calfj}{{\mathcal F}\!{\mathcal J}}

\newcommand{\calvcyc}{{\mathcal V}{\mathcal C}{\mathcal Y}}

\newcommand{\cala}{{\mathcal A}}

\newcommand{\calc}{{\mathcal C}}

\newcommand{\calf}{{\mathcal F}}

\newcommand{\calh}{{\mathcal H}}

\newcommand{\calp}{{\mathcal P}}

\newcommand{\cals}{{\mathcal S}}

\newcommand{\calu}{{\mathcal U}}

\newcommand{\IH}{{\mathbb H}}

\newcommand{\IP}{{\mathbb P}} 
\newcommand{\IQ}{{\mathbb Q}} 
\newcommand{\IR}{{\mathbb R}}

\newcommand{\IZ}{{\mathbb Z}}

\newcommand{\bfE}{{\mathbf E}}

\newcommand{\bff}{{\mathbf f}}

\newcommand{\bfi}{{\mathbf i}}

\newcommand{\bfK}{{\mathbf K}}  
 
\newcommand{\bfL}{{\mathbf L}}

\newcommand{\bfp}{{\mathbf p}}

\DeclareMathAlphabet{\matheurm}{U}{eur}{m}{n}

\newcommand{\bfWh}{{\mathbf W}{\mathbf h}}

\newcommand{\curs}{\EuR}

\newcommand{\Or}{\curs{Or}}

\newcommand{\SPECTRA}{\curs{SPECTRA}}

\newcommand{\ANR}{\operatorname{ANR}}

\newcommand{\asmb}{\operatorname{asmb}}
\newcommand{\aut}{\operatorname{aut}}

\newcommand{\BTop}{\operatorname{BTop}}

\newcommand{\CAT}{\operatorname{CAT}}

\newcommand{\zentrum}{\operatorname{center}}
\newcommand{\cd}{\operatorname{cd}}

\newcommand{\coker}{\operatorname{coker}}
\newcommand{\colim}{\operatorname{colim}}

\newcommand{\Diff}{\operatorname{Diff}}

\newcommand{\ev}{\operatorname{ev}}

\newcommand{\FP}{\operatorname{FP}}

\newcommand{\GL}{\operatorname{GL}}

\newcommand{\hocolim}{\operatorname{hocolim}}

\newcommand{\id}{\operatorname{id}}

\newcommand{\im}{\operatorname{im}}

\newcommand{\NK}{{N\!K}}

\newcommand{\Nil}{\operatorname{Nil}}

\newcommand{\Out}{\operatorname{Out}}
\newcommand{\PL}{\operatorname{PL}}

\newcommand{\pr}{\operatorname{pr}}

\newcommand{\supp}{\operatorname{supp}}

\newcommand{\Top}{\operatorname{Top}}
\newcommand{\TOP}{\operatorname{TOP}}

\newcommand{\Wh}{\operatorname{Wh}}

\newcommand{\pt}{\{\bullet\}}

     \newcounter{commentcounter}

     \theoremstyle{plain} \newtheorem{theorem}{Theorem}[section]
      \newtheorem{lemma}[theorem]{Lemma}
     
     \newtheorem{proposition}[theorem]{Proposition}
     \newtheorem{conjecture}[theorem]{Conjecture}
     
      \newtheorem*{theorem*}{Theorem}
     \newtheorem*{theoremA*}{Theorem A} \newtheorem*{theoremB*}{Theorem B}

     \theoremstyle{definition} \newtheorem{definition}[theorem]{Definition}
      \newtheorem{example}[theorem]{Example}
     \newtheorem{question}[theorem]{Question} 
      \newtheorem{remark}[theorem]{Remark}

     \newtheorem*{definition*}{Definition}

     \theoremstyle{remark}

     \makeatletter\let\c@equation=\c@theorem\makeatother

\newcommand{\EGF}[2]{E_{#2}(#1)}                   
 
\newcommand{\OrGF}[2]{\Or_{#2}(#1)}                

\newcommand{\higherlim}[3]{{\setbox1=\hbox{\rm lim}
        \setbox2=\hbox to \wd1{\leftarrowfill} \ht2=0pt \dp2=-1pt
        \mathop{\vtop{\baselineskip=5pt\box1\box2}}
        _{#1}}^{#2}#3}

\newcommand{\version}[1]                       
{\begin{center} last edited on #1\\
last compiled on \today \ at \DTMcurrenttime.\\
name of texfile: \jobname
\end{center}
}

\begin{document}

     \begin{abstract}
       This is a survey on the Farrell-Jones Conjecture about the algebraic $K$- and
       $L$-theory of groups rings and its applications to algebra, geometry, group theory, and topology.
     \end{abstract}

     \maketitle


  \typeout{------------------- Introduction -----------------}
  \section{Introduction}\label{sec:introduction}

  The Farrell-Jones Conjecture is devoted to the algebraic $K$- and $L$-theory of the
  group ring $RG$ of a group $G$ and a ring $R$. Certainly 
  $K$- and $L$-theory are rather sophisticated theories. Group
  rings are very difficult rings, for instance, they are in general not commutative,  Noetherian,
  or regular, and may have zero-divisors. So studying the algebraic
  $K$-theory and $L$-theory of group rings is hard and seems to be on the first glance a
  very special problem. So why should one care?

  The answer to this question is that information about the $K$- or $L$-theory of group
  rings has many striking applications to algebra, geometry, group theory, topology, and
  operator algebras and that meanwhile the Farrell-Jones Conjecture is known for a large
  class of groups. In the sequel we call a group a \emph{Farrell-Jones group} if it
  satisfies the so called \emph{Full Farrell-Jones Conjecture} which does imply all the
  variants of the Farrell-Jones Conjecture appearing in the literature. The class $\calfj$
  of Farrell-Jones groups contains many groups which ``occur in daily life'', e.g.,
  hyperbolic groups, finite dimensional CAT(0)-groups, lattices, fundamental groups of
  manifolds of dimension $\le 3$, and $S$-arithmetic groups. Moreover, it has very
  interesting inheritance properties, for instance it is closed under passing to subgroups
  and to directed colimits (with arbitrary structure maps).  In particular it contains all
  directed colimits (with arbitrary structure maps) of hyperbolic groups and hence a lot
  of groups with exotics properties, e.g., groups with expanders.

  There are many open conjectures in algebra, geometry, group theory, and topology which
  do follow from the Full Farrell-Jones Conjecture but for which the class for which they
  were proved by other methods is much smaller than the class $\calfj$. Here is a
  (incomplete) list of such conjectures or applications which we will discuss or for which
  we will give references for a discussion:

  \begin{itemize}
  \item Vanishing of projective class groups of integral group rings and of the Whitehead
    group for torsionfree groups;

   \item Classification of closed manifolds;

  \item Conjectures due to Bass, Borel, Novikov, and Serre;

  \item Poincar\'e duality groups;

  \item Automorphism groups of aspherical closed manifolds;

  \item Hyperbolic groups with spheres as boundary;

  \item The stable version of the Cannon Conjecture.

  \end{itemize}
  
  So for the reader the applications are more accessible and interesting than to figure
  out the complicated formulation of the Full-Farrell-Jones Conjecture itself. Therefore the reader
  may concentrate on
  Section~\ref{sec:The_projective_class_group}, Section~\ref{sec:The_Whitehead_group},~%
  Section~\ref{sec:Further_applications_of_the_farrell-Jones_Conjecture__for_torsionfree_groups},~%
    Subsection~\ref{subsec:Farrell-Jones-groups}, Section~\ref{sec:Further_applications_of_the_Full_Farrell-Jones_Conjecture},
  and Section~\ref{sec:Can_the_Full_Farrell-Jones_Conjecture_be_true_for_all_groups?}.
     
    The original formulation of the Farrell-Jones Conjecture with rings as coefficients
    appears in~\cite[1.6 on page~257]{Farrell-Jones(1993a)}.

    This article can be viewed as an extract of the book~\cite{Lueck(2022book)}, where much
    more information and references to the literature about the Farrell-Jones Conjecture
    are given.


\subsection{Acknowledgments}\label{subsec:Acknowledgements}

The paper is  funded  by the Deutsche
Forschungsgemeinschaft (DFG, German Research Foundation) under Germany's Excellence
Strategy \--- GZ 2047/1, Projekt-ID 390685813, Hausdorff Center for Mathematics at Bonn.

The author  thanks the referees for their useful comments.

The paper is organized as follows:
\tableofcontents


\typeout{------------------ Section 2:  Projective class group   -----------------------}

\section{The projective class group}%
\label{sec:The_projective_class_group}


\subsection{The definition  of the projective class group}%
\label{subsec:The_definition_of_the_projective_class_group}

The \emph{projective class group $K_0(S)$ of a ring $S$} is the abelian group which is
obtained from the Grothendieck construction applied to the abelian semigroup of isomorphism
classes of finitely generated projective $S$-modules under direct sum. Equivalently, it
can be described as the abelian group whose generators are isomorphism classes $[P]$ of
finitely generated projective $S$-modules $P$ and for every exact sequence
$0 \to P_0 \to P_1 \to P_2 \to 0$ of finitely generated projective $S$-modules we require
the relation $[P_1] = [P_0] + [P_2]$.

The second definition yields the following universal property of the projective class
group. Note  that any finitely generated projective $S$-module $P$ defines an element $[P]$ in
$K_0(S)$.  An \emph{additive invariant} $(A,a)$ for the category of finitely generated projective
$S$-modules consists of an abelian group $A$ and an assignment which associates to every
finitely generated projective $S$-module $P$ an element $a(P) \in A$ such that for every
exact sequence $0 \to P_0 \to P_1 \to P_2 \to 0$ of finitely generated projective
$S$-modules we have $a(P_1) = a(P_0) + a(P_2)$. Then $K_0(S)$ together with the assignment
$P \mapsto [P]$ is the \emph{universal additive invariant} for the category of finitely generated
projective $S$-modules in the sense that for any additive invariant $(A,a)$ there is
precisely one homomorphism of abelian groups $\Phi \colon K_0(S) \to A$  such that
$\Phi([P]) = a(P)$ holds for any finitely generated projective $S$-module $P$.

The \emph{reduced projective class group $\widetilde{K}_0(S)$ of a ring $S$} is obtained
from $K_0(S)$ by dividing out the subgroup generated by all finitely generated free
$S$-modules.  Any finitely generated projective $S$-module $P$ defines an element $[P]$ in
$K_0(S)$ and hence also a class $[P]$ in $\widetilde{K}_0(S)$. The decisive property of
$\widetilde{K}_0(S)$ is that $[P] = 0$ holds in $\widetilde{K}_0(S)$ if and only if $P$ is
\emph{stably free}, i.e., there are natural numbers $m$ and $n$ satisfying
$P \oplus S^m \cong_S S^n$. So roughly speaking, $[P] \in \widetilde{K}_0(S)$ measures the
deviation of a finitely generated projective $S$-module $P$ from being stably free.


\subsection{Wall's finiteness obstruction}%
\label{subsec:Wall's_finiteness_obstruction}

A $CW$-complex $X$ is called \emph{finitely dominated} if there is a finite $CW$-complex
$Y$ and maps $i \colon X \to Y$ and $r \colon Y \to X$ such that $r \circ i$ is homotopic
to $\id_X$. Often one can construct a finitely dominated $CW$-complex with interesting
properties but needs to know whether it is homotopy equivalent to a finite
$CW$-complex. This problem is decided by the \emph{finiteness obstruction} due to  Wall. A
finitely dominated connected $CW$-complex $X$ with fundamental group $\pi$ determines an
element $\widetilde{o}(X) \in \widetilde{K}_0(\IZ \pi)$, which vanishes if and only if $X$
is homotopy equivalent to a finite $CW$-complex.  So it is interesting to know whether
$\widetilde{K}_0(\IZ \pi)$ vanishes because then $\widetilde{o}(X)$ is automatically
trivial.

The question whether a finitely dominated $CW$-complex is homotopy equivalent to a finite
$CW$-complex appears naturally in the construction of closed manifolds with certain
properties, since a closed manifold is homotopy equivalent to a finite $CW$-complex, and
one may be able to construct  a finitely dominated $CW$-complex as a first approximation
up to homotopy.  The \emph{Spherical Space
  Form Problem} is a prominent example. It aims at
the classification of closed manifolds whose universal coverings are diffeomorphic or
homeomorphic to the standard sphere.

One can actually show for a finitely presented group $G$ that
$\widetilde{K}_0(\IZ G)$ vanishes if and only if every finitely dominated connected
$CW$-complex with fundamental group isomorphic to $G$ is homotopy equivalent to a finite
$CW$-complex. So we have an algebraic assertion and a topological assertion for any finitely presented group
$G$ which turn out to be equivalent.


\subsection{Conjectures about the projective class group for  torsionfree groups}%
\label{subsec:Conjectures_about_the_projective_class_group_for_torsionfree_groups}

\begin{conjecture}\label{con:K_0(ZG)_vanishes_for_torsionfree_G}
  The reduced projective class group  $\widetilde{K}_0(\IZ G)$ vanishes if $G$ is torsionfree.
\end{conjecture}

A ring $S$ is called \emph{regular} if it is \emph{Noetherian}, i.e., every submodule of a
finitely generated $S$-module is finitely generated, and every $S$-module has a
finite-dimensional projective resolution. Since $\IZ$ is a regular ring,
Conjecture~\ref{con:K_0(ZG)_vanishes_for_torsionfree_G} is a special case of the following
more general conjecture.

\begin{conjecture}\label{con:K_0(ZG)_vanishes_for_torsionfree_G_and_regular_R}
  If $G$ is torsionfree  and $R$ is a regular ring, then the canonical map
  \[
  K_0(R) \xrightarrow{\cong}  K_0(RG)
\]
is bijective.

In particular $\widetilde{K}_0(RG)$ vanishes, if $G$ is torsionfree  and
$R$ is a principal ideal domain.
\end{conjecture}

\begin{conjecture}[Kaplansky]\label{con:Idempotent_Conjecture}
  If $R$ is a field and the group $G$ is torsionfree,
  then the group ring $RG$ has only trivial idempotents, namely, $0$ and $1$.
\end{conjecture}

A group $G$ is called of type \textup {(FF)} or \textup{(FP)} respectively, if the trivial
$\IZ G$-module $\IZ$ possesses a finite dimensional resolution consisting of finitely
generated free $\IZ G$-modules or finitely generated projective $\IZ G$-modules respectively.

\begin{conjecture}[Serre]\label{con:Serre}
  A group $G$ is of type \textup {(FF)} if it is of type \textup{(FP)}.
\end{conjecture}

Conjecture~\ref{con:Idempotent_Conjecture} follows from
Conjecture~\ref{con:K_0(ZG)_vanishes_for_torsionfree_G_and_regular_R}, whereas
Conjecture~\ref{con:Serre} follows from
Conjecture~\ref{con:K_0(ZG)_vanishes_for_torsionfree_G}.
Conjectures~\ref{con:K_0(ZG)_vanishes_for_torsionfree_G}
and~\ref{con:K_0(ZG)_vanishes_for_torsionfree_G_and_regular_R} are not true if one drops
the condition torsionfree.

For more information about the projective class group and the finiteness obstruction we refer for instance
to~\cite[Chapter~2]{Lueck(2022book)},~\cite{Wall(1965a)}, and~\cite{Wall(1966)}.


\typeout{----------------------- Section 3: The Whitehead group    -------------------}

\section{The Whitehead group }%
\label{sec:The_Whitehead_group}


\subsection{The definition  of $K_1(S)$ and of the Whitehead group}%
\label{subsec:The_definition_of_K_1(S)_and_of_the_Whitehead_group}

\begin{definition}[$K_1$-group $K_1(S)$]\label{def:K_1(S)}
Let $S$ be a ring. Define the \emph{$K_1$-group} $K_1(S)$
to be the abelian group whose generators
are conjugacy classes $[f]$ of automorphisms $f\colon  P \to P$ of finitely
generated projective $S$-modules with the following relations:
\begin{itemize}
\item Additivity\\
Given a commutative diagram of finitely generated projective $S$-modules
\[
\xymatrix{
0 \ar[r] & P_1 \ar[r]^{i} \ar[d]^{f_1} &  P_2 \ar[r]^{p} \ar[d]^{f_2}
& P_3 \ar[r] \ar[d]^{f_3} & 0
\\
0 \ar[r] & P_1 \ar[r]^{i} &  P_2 \ar[r]^{p} & P_3 \ar[r] & 0
}
\]
with exact rows and automorphisms as vertical arrows, we get
\[
[f_1] + [f_3] = [f_2];
\]

\item Composition formula\\
Given automorphisms $f,g \colon  P \to P$ of a finitely generated projective
$S$-module $P$, we get
\[
[g \circ f] = [f] + [g].
\]
\end{itemize}
\end{definition}

One should view $K_1(S)$ together with the assignment sending an
automorphism $f \colon P \to P$ of a finitely generated projective
$S$-module $P$ to its class $[f] \in K_1(S)$ as the
\emph{universal determinant}.
Namely, for any abelian group $A$ and
assignment $a$ which  sends the automorphism $f$ of a finitely generated
projective $S$-module to $a(f) \in A$ such that $(A,a)$ satisfies
additivity and the composition formula appearing in
Definition~\ref{def:K_1(S)}, there exists precisely one homomorphism
of abelian groups $\phi \colon K_1(S) \to A$ such that $\phi([f]) =
a(f)$ holds for every automorphism $f$ of a finitely generated
projective $S$-module.

Define the abelian group $K_1^f(S)$
analogous to $K_1(S)$ but with finitely generated projective
replaced by finitely generated free everywhere. Then the canonical homomorphism
\[
\alpha \colon K_1^f(S) \xrightarrow{\cong} K_1(S), \quad [f]~\mapsto~[f]
\]
is an isomorphism.

An equivalent definition of $K_1(S)$ of a ring $S$ is the abelianization of the general
linear group $\GL(S) = \colim_{n \to \infty} \GL_n(S)$.

We always have the following map of abelian groups
\[
i \colon S^{\times}/[S^{\times},S^{\times}] \to K_1(S),
\quad [x] \mapsto [r_x \colon S \to S]
\]
where $r_x$ denotes right multiplication  with $x$. It is neither injective nor
surjective in general.  In particular any unit in $S$ defines an element in $K_1(S)$.

If $G$ is a group, a \emph{trivial unit} in $\IZ G$ is given by an element of the form
$\epsilon \cdot g$ for $\epsilon \in \{\pm 1\}$ and $g \in G$.

\begin{definition}[Whitehead group]\label{def:Whitehead_group}
    Given a group $G$, the \emph{Whitehead
  group} $\Wh(G)$ is  the quotient of $K_1(\IZ G)$ by the subgroup generated by trivial
units.
\end{definition}


\subsection{The $s$-Cobordism Theorem}\label{subsec:The_s-Cobordism_Theorem}

Given a closed manifold $M$, an \emph{$h$-cobordism $W$ over $M$} is a compact manifold
$W$ such that its boundary $\partial W$ can be written as a disjoint union
$\partial W = \partial_0W \amalg \partial_1 W$, there is a preferred identification of $M$ with
$\partial_0W$, and the inclusions $\partial_k W \to W$ are homotopy equivalences for
$k = 0,1$.

\begin{theorem}[The $s$-Cobordism Theorem]\label{the:s-Cobordism_Theorem}
 The set of isomorphism classes relative $M$  of $h$-cobordisms over $M$ can be
identified with $\Wh(\pi)$ if $M$ is connected, has dimension $\ge 5$, and $\pi$
denotes its fundamental group.
\end{theorem}

Theorem~\ref{the:s-Cobordism_Theorem}  is astonishing, since the set of isomorphism classes of
$h$-cobordism relative $M$ over $M$ a priori depends on $M$, whereas $\Wh(\pi)$ depends
only on the fundamental group.  In the classification of closed manifolds it is often a
key step to decide whether an \emph{$h$-cobordism $W$ over $M$} is trivial, i.e,
isomorphic relative $M$ to $M \times [0,1]$, since this has the consequence that $M$ and
$\partial_1 W$ are isomorphic.  It is not hard to show that $\Wh(\{1\})$ is trivial which
implies together with the results above the Poincar\'e Conjecture in dimensions $\ge 5$
which says  that a closed manifold which is homotopic to $S^n$ is homeomorphic to $S^n$.

One can show for a finitely presented group $G$ and any natural number $n \ge 5$ that
$\Wh(G)$ is trivial if and only if for every connected $n$-dimensional closed
manifold $M$ with fundamental group isomorphic to $G$ every $h$-cobordisms over $M$ is
trivial, where one can work in the smooth, PL, or topological category.
So we have again an algebraic assertion and a topological assertion
for  any finitely presented group $G$ which turn out to be equivalent.

For more information and references about Theorem~\ref{the:s-Cobordism_Theorem}
we refer to~\cite[Chapter~2]{Lueck-Macko(2024)}.
Note  that it is a corner stone in the so called \emph{Surgery Program} aiming at the classification of manifolds.


\subsection{Simple homotopy equivalences}\label{subsec:Simple_homotopy_equivalences}

We have the inclusion of spaces
$S^{n-2}\subset S^{n-1}_+ \subset S^{n-1} \subset D^n$
where $S^{n-1}_+ \subset S^{n-1}$ is the upper hemisphere. The pair
$(D^n,S^{n-1}_+)$ carries an obvious relative $CW$-structure. Namely,
attach an $(n-1)$-cell to $S^{n-1}_+$ by the attaching map $\id\colon  S^{n-2} \to
S^{n-2}$ to obtain $S^{n-1}$. Then we attach to $S^{n-1}$ an $n$-cell by the
attaching map $\id\colon  S^{n-1} \to S^{n-1}$ to obtain $D^n$.
Let $X$ be a $CW$-complex. Let $q\colon  S^{n-1}_+ \to X$ be a map satisfying
$q(S^{n-2}) \subset X_{n-2}$ and $q(S^{n-1}_+) \subset X_{n-1}$.
Let $Y$ be the space $D^n \cup_{q} X$, i.e., the pushout
\[
  \xymatrix{S^{n-1}_+ \ar[r]^q \ar[d]_i
    &
    X \ar[d]^{\overline{i}}
    \\
    D^n \ar[r]_{\overline{q}}
    &
    Y
    }
\]
where $i$ is the inclusion.
Then $Y$ inherits a $CW$-structure by putting $Y_k = \overline{i}(X_k)$
for $k \le n-2$, $Y_{n-1} = \overline{i}(X_{n-1}) \cup \overline{q}(S^{n-1})$ and
$Y_k = \overline{i}(X_k) \cup \overline{q}(D^n)$ for $k \ge n$.
Note that $Y$ is obtained from $X$ by attaching one $(n-1)$-cell
and one $n$-cell. Since the map $i\colon  S^{n-1}_+ \to D^n$ is a homotopy
equivalence and a cofibration, the map $\overline{i}\colon  X \to Y$ is a homotopy equivalence
and a cofibration.
We call $\overline{i}$ an \emph{elementary expansion}
and say that $Y$ is obtained from $X$ by an elementary expansion.
There is a map $r\colon  Y \to X$ with $r \circ \overline{i} = \id_X$.
The map $r$ is unique up to homotopy relative $\overline{i}(X)$. We call any such map $r$
an \emph{elementary collapse}
and say that $X$ is obtained from $Y$ by an elementary collapse.

The next definition is due to JHC Whitehead~\cite{Whitehead(1950)}.

\begin{definition}[Simple homotopy equivalence]%
\label{def:simple_homotopy_equivalence}
Let $f\colon  X \to Y$ be a map of finite $CW$-complexes. We call it
a simple \emph{homotopy equivalence}
if there is a sequence of maps
\[
X = X[0] \xrightarrow{f_0} X[1] \xrightarrow{f_1} X[2] \xrightarrow{f_2}  \cdots
\xrightarrow{f_{n-1}} X[n] = Y
\]
such that each $f_i$ is an elementary expansion or elementary collapse and
$f$ is homotopic to the composite of the maps $f_i$.
\end{definition}

The idea of the definition of a simple homotopy equivalence is that
such a map can be written as a composite of elementary maps, namely, elementary
expansions and collapses,  which
are obviously homotopy equivalences and in some sense the smallest and
most elementary steps to pass from one finite $CW$-complex to another
without changing the homotopy type. If one works with simplicial
complexes, an elementary map has a purely combinatorial description.
An elementary collapse means to delete the interior of one  simplex and one of its faces
that is not shared by another simplex.
So one can describe the passage from one finite simplicial complex to
another coming from a simple homotopy equivalence by finitely many
combinatorial data. 

This approach is similar to the idea in knot theory that two knots are
equivalent if one can pass from one knot to the other by a sequence of
elementary moves, the so-called Reidemeister moves. A Reidemeister
move obviously does not change the equivalence class of a knot and,
indeed, it turns out that one can pass from one knot to a second knot
by a sequence of Reidemeister moves if and only if the two knots are equivalent,
The analogous statement is not true for
homotopy equivalences $f\colon X \to Y$ of finite $CW$-complexes
because there is an obstruction for $f$ to be simple, namely, its
Whitehead torsion $\tau(f)$ which is an element in $\Wh(\pi_1(Y))$.

\begin{theorem}[Whitehead torsion and simple homotopy equivalences]%
\label{the:simple_and_tau(f)_vanishes} \
  \begin{enumerate}
  \item\label{the:simple_and_tau(f)_vanishes:obstruction} Let $f \colon X \to Y$ be a
    homotopy equivalence of connected finite $CW$-complexes. Then its Whitehead torsion
    $\tau(f) \in \Wh(\pi_1(Y))$ vanishes if and only if $f$ is a simple homotopy
    equivalence;

  \item\label{the:simple_and_tau(f)_vanishes:realization} Let $Y$ be a connected finite
    $CW$-complex.  Then for any element $u \in \Wh(\pi_1(Y))$ there exists a homotopy
    equivalence $f\colon X \to Y$ with a connected finite $CW$-complex $X$ as source
    satisfying $\tau(f) = u$.
\end{enumerate}
\end{theorem}

Theorem~\ref{the:simple_and_tau(f)_vanishes} implies for any finitely presented group $G$
and any connected finite $CW$-complex $Y$ with $\pi_1(Y) \cong G$ that $\Wh(G)$ is trivial
if and only if every homotopy equivalence $f\colon X \to Y$ with a connected finite
$CW$-complex $X$ as source and $Y$ as target is simple.

We mention the deep theorem due to Chapman~\cite{Chapman(1973)} that every  (not necessarily cellular) homeomorphism
of finite $CW$-complexes is simple.

For more information about the Whitehead group and simple homotopy equivalences
 we refer for instance~\cite{Cohen(1973)} and~\cite[Chapter~2]{Lueck(2022book)}.


\subsection{Conjecture about the Whitehead group  for  torsionfree groups}%
\label{subsec:Conjecture_about_the_Whitehead_group_for_torsionfree_groups}

\begin{conjecture}[Farrell-Jones Conjecture for $\Wh(G)$ for torsionfree $G$]%
\label{con:FHC_Farrell-Jones_Conjecture_for_Wh(G)_for_torsionfree_Groups}%
Let $G$ be a torsionfree group. Then $\Wh(G)$ vanishes.
\end{conjecture}


\typeout{----------------------- Section 4: Lower K-theory  -------------------}

\section{Lower $K$-theory}\label{sec:Lower_K-goups}


\subsection{The Bass-Heller-Swan decomposition}\label{subsec:The_Bass-Heller-Swan_decomposition}

In the sequel we write $R[\IZ]$ as the ring $R[t,t^{-1}]$ of finite Laurent polynomials in
$t$ with coefficients in $R$.  Obviously the ring $R[t]$ of polynomials in $t$ with
coefficients in $R$ is a subring of $R[t,t^{-1}]$. Define the ring homomorphisms
\[
\begin{array}{lll}
\ev_0 \colon R[t]~\to~R, & & \sum_{n \in \IZ} r_n t^n~\mapsto~r_0;
\\
i' \colon R~\to~R[t], & &  r~\mapsto~r\cdot t^0;
\\
i \colon R~\to~R[t,t^{-1}], & &  r~\mapsto~r\cdot t^0.
\end{array}
\]

\begin{definition}[$\NK_n(R)$]\label{def:NK_n(R)}
Define for $n = 0,1$
\[
\NK_n(R)  := \ker\bigl((\ev_0)_* \colon  K_n(R[t]) \to K_n(R)\bigr).
\]
\end{definition}

Recall that an endomorphism $f \colon P \to P$ of an  $R$-module
$P$ is called \emph{nilpotent}  if there exists a positive
integer $n$ with $f^n = 0$.

\begin{definition}[Nil-group $\Nil_0(R)$]
Define the \emph{$0$-th Nil-group} $\Nil_0(R)$
to be the abelian group whose generators are
conjugacy classes $[f]$ of nilpotent endomorphisms $f\colon P \to P$ of
finitely generated projective $R$-modules with the following relations.
Given a commutative diagram of finitely generated projective
$R$-modules
\[
\xymatrix{ 0 \ar[r] & P_1 \ar[r]^{i} \ar[d]^{f_1} & P_2 \ar[r]^{p}
  \ar[d]^{f_2} & P_3 \ar[r] \ar[d]^{f_3} & 0
  \\
  0 \ar[r] & P_1 \ar[r]^{i} & P_2 \ar[r]^{p} & P_3 \ar[r] & 0 }
\]
with exact rows and nilpotent endomorphisms as vertical arrows, we
get
\[
[f_1] + [f_3] = [f_2].
\]
\end{definition}

Let $\iota \colon K_0(R) \to \Nil_0(R)$ be the homomorphism sending
the class $[P]$ of a finitely generated projective $R$-module $P$ to
the class $[0 \colon P \to P]$ of the trivial endomorphism of $P$.

\begin{definition}[Reduced Nil-group $\widetilde{\Nil}_0(R)$]
Define the \emph{reduced  $0$-th Nil-group}
$\widetilde{\Nil}_0(R)$ to be  the cokernel of the map $\iota$.
\end{definition}

The homomorphism $\Nil_0(R) \to K_0(R), \; [f \colon P \to P] \mapsto [P]$ is
a retraction of the map $\iota$.  So we get a natural
splitting
\[
\Nil_0(R)~\xrightarrow{\cong}~\widetilde{\Nil}_0(R) \oplus K_0(R).
\]

Denote by
\[
j \colon \NK_1(R) \to K_1(R[t])
\]
the inclusion. Let
\[
l_{\pm} \colon R[t] \to R[t,t^{-1}]
\]
be the inclusion of rings sending $t$ to $t^{\pm 1}$. Define
\[
j_{\pm} := K_1(l_{\pm}) \circ j \colon \NK_1(R) \to K_1(R[t,t^{-1}]).
\]
The homomorphism
\[
B \colon K_0(R) \to K_1(R[t,t^{-1}])
\]
sends the class $[P]$ of a
finitely generated projective $R$-module $P$ to the class
$[r_t \otimes_R \id_P]$ of the $R[t,t^{-1}]$-automorphism
$r_t \otimes_R \id_P \colon R[t,t^{-1}] \otimes_R P \to R[t,t^{-1}] \otimes_R P$
that maps $u \otimes p$ to $ut \otimes p$.  The homomorphism
\[
N' \colon \Nil_0(R) \to K_1(R[t])
\]
sends the class $[f]$ of the nilpotent
endomorphism $f \colon P \to P$ of the finitely generated projective
$R$-module $P$ to the class $[\id - r_{t} \otimes_R f]$ of the
$R[t]$-automorphism
\[
\id - r_{t} \otimes_R f \colon R[t] \otimes_R P \to R[t]
\otimes_R P, \quad 
u \otimes p~\mapsto~u \otimes p - ut \otimes f(p).
\]
This is indeed an automorphism. Namely, if $f^{n+1} = 0$, then an
inverse is given by $\sum_{k=0}^{n} \left(r_{t} \otimes_R f\right)^k$.
The composite  of $N'$ with both $(\ev_0)_* \colon K_1(R[t])
\to K_1(R)$ and $\iota \colon K_0(R) \to \Nil_0(R)$ is trivial.  Hence
$N'$ induces a homomorphism
\[
N \colon \widetilde{\Nil}_0(R) \to \NK_1(R).
\]
The proof of the following theorem can be found for instance 
in~\cite{Bass-Heller-Swan(1964)} (for regular rings),~\cite[Chapter XII]{Bass(1968)},~%
\cite[Theorem~3.2.22 on page~149]{Rosenberg(1994)},
and~\cite[3.6 in Section~III.3 on page~205]{Weibel(2013)}.

\begin{theorem}[Bass-Heller-Swan decomposition for $K_1$]%
\label{the:Bass-Heller-Swan_decomposition_for_K_1}%
The following maps are
isomorphisms of abelian groups, natural in $R$,
\begin{eqnarray*}
& N \colon  \widetilde{\Nil}_0(R)  \xrightarrow{\cong} \NK_1(R);&
\\
& j \oplus K_1(i') \colon  \NK_1(R) \oplus K_1(R)  
\xrightarrow{\cong}  K_1(R[t]);&
\\
&B \oplus K_1(i) \oplus j_+ \oplus j_- \colon 
 K_0(R) \oplus K_1(R) \oplus \NK_1(R) \oplus \NK_1(R)
 \xrightarrow{\cong}  K_1(R[t,t^{-1}]).&
\end{eqnarray*}
\end{theorem}

One easily checks that Theorem~\ref{the:Bass-Heller-Swan_decomposition_for_K_1}
applied to $R = \IZ G$ implies the following reduced version

\begin{theorem}[Bass-Heller-Swan decomposition for $\Wh(G \times \IZ)$]%
\label{the:Bass-Heller-Swan_decomposition_for_Whitehead_groups}
Let $G$ be a group. Then there is an isomorphism of abelian groups, natural in $G$
\[
\overline{B} \oplus  \Wh(i) \oplus \overline{j_+} \oplus \overline{j_-} \colon
\widetilde{K}_0(\IZ G) \oplus \Wh(G) \oplus \NK_1(\IZ G) \oplus \NK_1(\IZ G)
 \xrightarrow{\cong}  \Wh(G \times \IZ).
\]
\end{theorem}

 The proof of the next result can be found for instance
    in~\cite[Exercise~3.2.25 on page~152]{Rosenberg(1994)} 
    or~\cite[Corollary~16.5 on page~226]{Swan(1968)}.

\begin{theorem}[Bass-Heller-Swan decomposition for $K_1$ for regular rings]%
\label{the:Bass-Heller-Swan_decomposition_for_K_1_for_regular_rings}
Suppose that $R$ is  regular. Then we get
\[
\widetilde{\Nil}_0(R) = \NK_1(R) =0,
\]
and the Bass-Heller-Swan decomposition
of~Theorem~\ref{the:Bass-Heller-Swan_decomposition_for_K_1} reduces
to the isomorphism
\[
B \oplus i_* \colon K_0(R) \oplus K_1(R) \xrightarrow{\cong}
K_1(R[t,t^{-1}]).
\]
\end{theorem}

The next theorem is a consequence of
Theorem~\ref{the:Bass-Heller-Swan_decomposition_for_K_1}.

\begin{theorem}[Fundamental Theorem of $K$-theory in dimension $1$]%
 \label{the:Fundamental_Theorem_of_K-theory_in_dimension_1}
  There is a  sequence which is natural in $R$ and exact  
\begin{multline*}
  0 \to K_1(R) \xrightarrow{K_1(k_+) \oplus -K_1(k_-)} K_1(R[t]) \oplus K_1(R[t^{-1}])
  \\
  \xrightarrow{K_1(l_+)_* \oplus  K_1(l_-)}  K_1(R[t,t^{-1}])  \xrightarrow{C} K_0(R) \to 0,
\end{multline*}
where $k_+$, $k_-$, $l_+$, and $l_-$ are the obvious inclusions.

If we regard it as an acyclic $\IZ$-chain complex, there exists a chain contraction, natural in $R$.
\end{theorem}


\subsection{Negative $K$-groups}\label{subsec:Negative_K-goups}

Recall that we get from
Theorem~\ref{the:Fundamental_Theorem_of_K-theory_in_dimension_1} an
isomorphism
\[
K_0(R) = \coker\left(K_1(R[t]) \oplus K_1(R[t^{-1}])\to K_1(R[t,t^{-1}])\right).
\]
This motivates the following definition of negative $K$-groups due to
Bass.

\begin{definition}\label{def:negative_K-theory}%
Given a ring $R$, define inductively for
  $n = -1, -2, \ldots$
\[
K_n(R)
:= \coker\left(K_{n+1}(R[t]) \oplus K_{n+1}(R[t^{-1}])
\to K_{n+1}(R[t,t^{-1}])\right).
\]
Define for $n = -1, -2, \ldots$
\[
\NK_n(R)
:= \coker\left(K_n(R) \to K_n(R[t])\right).
\]
\end{definition}

With this definition one obtains the following  obvious extension of
Theorem~\ref{the:Bass-Heller-Swan_decomposition_for_K_1} to  lower $K$-theory.

\begin{theorem}[Bass-Heller-Swan decomposition for middle and lower $K$-theory]%
\label{the:Bass-Heller-Swan-decomposition_for_middle_and_lower_K-theory}%
  There  are  isomorphisms of abelian groups, natural in $R$, for 
  $n = 1, 0, -1,  -2, \ldots $
  \begin{eqnarray*}
    & \NK_n(R) \oplus K_n(R)  \xrightarrow{\cong}  K_n(R[t]);&
    \\
    &K_n(R) \oplus K_{n-1}(R) \oplus \NK_n(R) \oplus \NK_n(R)
    \xrightarrow{\cong}  K_n(R[t,t^{-1}]).&
  \end{eqnarray*}
  There is a  sequence which is natural in $R$ and exact  for $n = 1,0, -1, \ldots$
  \begin{multline*}
    0 \to K_n(R) \xrightarrow{K_n(k_+) \oplus -K_n(k_-)} K_n(R[t]) \oplus K_n(R[t^{-1}])
    \\
    \xrightarrow{K_n(l_+) \oplus K_n(l_-)_*} K_n(R[t,t^{-1}])  \xrightarrow{C_n} K_{n-1}(R) \to 0,
      \end{multline*}
  where $k_+$, $k_-$, $l_+$, and $l_-$ are the obvious inclusions.
  
   If we regard it as an acyclic $\IZ$-chain complex, there exists a chain contraction, natural in $R$.
 \end{theorem}


\subsection{The bounded $s$-Cobordism Theorem}\label{subsec:The_bounded_s-Cobordism_Theorem}

Next we consider \emph{manifolds $W$ parametrized over} $\IR^k$, i.e., manifolds that are
equipped with a surjective proper map $p \colon W \to \IR^k$. Then there is the notion of
a \emph{bounded h-cobordism}. We state without giving further explanations the following
result whose proofs can be found in~\cite{Pedersen(1986)}
and~\cite[Appendix]{Weiss-Williams(1988)} and which contains the $s$-Cobordism
Theorem~\ref{the:s-Cobordism_Theorem} as a special case.

\begin{theorem}[Bounded $s$-Cobordism Theorem]\label{the:bounded_s-cobordism_theorem}
    Suppose that $M_0$ is a connected manifold 
  parametrized over $\IR^k$ satisfying  $\dim M_0 \geq 5$. Let $\pi$
  be its fundamental group.

  The  equivalence classes of bounded
  $h$-cobordisms over $M_0$ modulo bounded diffeomorphism relative $M_0$
  correspond bijectively to elements in $\kappa_{1-k} (\pi)$ where
\[
\kappa_{1-k} ( \pi ) =
\begin{cases} \Wh (\pi ) & \mbox{if } k=0;
  \\
  \widetilde{K}_0 ( \IZ \pi ) & \mbox{if } k=1;
  \\
  K_{1-k} ( \IZ \pi ) & \mbox{if } k \geq 2.
\end{cases}
\]
\end{theorem}


\subsection{Conjecture about middle and negative $K$-groups for  torsionfree groups}%
\label{subsec:Conjecture_about_middle_and_negative_K_groups_for_torsionfree_groups}

\begin{conjecture}[The Farrell-Jones Conjecture for negative $K$-theory and regular
  coefficient rings and torsionfree groups]%
  \label{con:FJC_neg_K-Theory_and_regular_coefficient_rings_and_torsionfree_groups}
  Let $R$ be a regular ring and $G$ be a torsionfree group. Then we get
  \[
    K_n(RG) = 0 \quad \text{ for } \quad n \le -1.
  \]
\end{conjecture}

Combing Conjectures~\ref{con:K_0(ZG)_vanishes_for_torsionfree_G},~%
\ref{con:FHC_Farrell-Jones_Conjecture_for_Wh(G)_for_torsionfree_Groups},
and~\ref{con:FJC_neg_K-Theory_and_regular_coefficient_rings_and_torsionfree_groups} for $R = \IZ$ yields

\begin{conjecture}[The Farrell-Jones Conjecture for middle and negative $K$-theory  for torsionfree groups with integral coefficients]%
\label{con:The_Farrell-Jones_Conjecture_for_middle_and_negative_K-theory}
Let $G$ be a torsionfree group. Then $\Wh(G)$, $\widetilde{K}_0(\IZ G)$, and $K_n(\IZ G)$ for $n \le -1$ vanish.
\end{conjecture}


\typeout{---------- Section 5:  Higher algebraic $K$-theory    ---------------}

\section{Higher algebraic $K$-theory}\label{sec:Higher_algebraic_K-theory}

Let $S$ be a ring. So far the algebraic $K$-groups $K_n(S)$ for $n \le 1$ have been
described in a purely algebraic fashion by generators and relations. This is also possible
for $K_2(S)$. The definition of the higher algebraic $K$-groups $K_n(S)$ for $n \ge 3$ has
been achieved topologically, namely, one assigns to a ring $S$ a space $K(S)$ and defines
$K_n(S)$ by the $n$-th homotopy group $\pi_n(K(S))$ for $n \ge 0$.  This will coincide
with the previous definition for $n = 0,1$. There are various definitions of the space
$K(S)$ that extend to more general settings, for instance the plus-construction and the
$Q$-construction for exact categories of Quillen, or the  $S_{\bullet}$-construction due to Waldhausen for
categories with cofibrations and weak equivalences. Algebraic $K$-groups of rings have some important useful
properties, such as compatibility with finite products, Morita equivalence, and certain
theorems such as the Resolution Theorem, the Devissage Theorem, the Localization Theorem,
the Cofinality Theorem, the Fibering Theorem, and the Approximation Theorem hold.


\subsection{The non-connective $K$-theory spectrum}%
\label{subsec:The_non-connective_K-theory_spectrum}

It is necessary  for the Farrell-Jones Conjecture to take into
account the negative $K$-groups as well. Therefore we will consider the non-connective
$K$-theory spectrum $\bfK(S)$ of a ring $S$ as constructed for instance by
Pedersen-Weibel~\cite{Pedersen-Weibel(1985)}, Schlichting~\cite{Cardenas-Pedersen(1997)},
or L\"uck-Steimle~\cite{Lueck-Steimle(2014delooping)} and define for $n \in \IZ$
\begin{eqnarray}
  K_n(S) = \pi_n(\bfK(S)).
  \label{K_n(S)_for_n_in_Z}
\end{eqnarray}
Of course this is consistent with the previous definitions.
Moreover, the Bass-Heller-Swan decomposition for middle and lower $K$-theory, see
Theorem~\ref{the:Bass-Heller-Swan-decomposition_for_middle_and_lower_K-theory},
extends to all $n \in \IZ$.

\begin{remark}[Glimpse of a homological behavior of $K$-theory]%
  \label{rem:Glimpse_of_a_homological_behavior_of_K-theory}
  In the case that $R$ is a regular ring, the Bass-Heller-Swan decomposition shades some
  homological flavour on $K$-theory.  Namely, observe the analogy between the two formulas
  \begin{eqnarray*}
    K_n(R[\IZ]) & \cong & K_{n-1}(R[\{1\}]) \oplus K_n(R[\{1\}]);
    \\
    H_n(B\IZ;A) & \cong & H_{n-1}(B\{1\};A) \oplus H_n(B\{1\};A),
  \end{eqnarray*}
  where in the second line we consider group homology with coefficients in some abelian
  group $A$, which corresponds to the role of $R$ in the first line. Analogously one gets
  for two torsionfree groups $G_0$ and $G_1$ and a regular ring $R$
  \begin{eqnarray*}
    \widetilde{K}_n(R[G_0 \ast G_1]) & \cong & \widetilde{K}_{n}(R[G_0]) \oplus \widetilde{K}_n(R[G_1]);
    \\
    \widetilde{H}_n(B(G_0 \ast G_1);A) & \cong & \widetilde{H}_n(BG_0;A) \oplus \widetilde{H}_n(BG_1;A),
  \end{eqnarray*}
  where $\widetilde{K}_n(RG)$ is the cokernel of $K_n(R) = K_n(R[\{1\}]) \to K_n(RG)$ and
  $\widetilde{H}_n(BG;A)$ is the cokernel of
  $\widetilde{H}_n(B\{1\};A) = H_n(\pt,A) \to H_n(BG;A)$.

Here and elsewhere  $\pt$ denotes the space consisting of one point. The space $BG$
is the \emph{classifying space of the group $G$},
which  is  up to homotopy  characterized by the property that it is a 
$CW$-complex with $\pi_1(BG) \cong G$ whose universal covering
is contractible. 
\end{remark}


\subsection{Conjecture about $K$-groups for  torsionfree groups and regular rings}%
\label{subsec:Conjecture_about_K-groups_for_torsionfree_groups_and_regular_rings}

We have already explained in
Remark~\ref{rem:Glimpse_of_a_homological_behavior_of_K-theory} that the assignment sending
a torsionfree group $G$ to $K_n(RG)$ for a regular ring has some homological behaviour.
This leads to the next conjecture.

\begin{conjecture}[Farrell-Jones Conjecture for $K$-theory for torsionfree groups and regular rings]%
\label{con:FJC_torsionfree_regular_higher_K-theory}%
Let $G$ be a torsionfree group. Let $R$ be a regular
  ring.  Then the so-called \emph{assembly map}
  \[
  H_n(BG;\bfK (R)) \to K_n(RG)
  \]
  is an isomorphism for $n \in \IZ$.
\end{conjecture}

Here $H_*(-;\bfK (R))$ denotes the homology theory 
that  is associated to the (non-connective) $K$-spectrum $\bfK (R)$. Recall that
$H_n(\pt;\bfK(R))$ is $K_n(R)$ for  $n \in \IZ$.

If one drops in Conjecture~\ref{con:FJC_torsionfree_regular_higher_K-theory} the condition
that $R$ is regular or the condition that $G$ is torsionfree, it becomes false. Later we
state a version of the Farrell-Jones Conjecture, where no conditions on $R$ or $G$ will
appear, see Section~\ref{sec:The_Full_Farrell-Jones_Conjecture}.

Using the Atiyah-Hirzebruch spectral sequence and the fact that for a regular
ring we have $K_n(R) = 0$ for $n \le -1$ and $\NK_n(R) = 0$ for $n \in \IZ$ and for a
principal ideal domain $R$ we get an isomorphism $\IZ \xrightarrow{\cong} K_0(R)$ by
sending $n \in \IZ$ to $n \cdot [R]$, one easily checks that
Conjecture~\ref{con:FJC_torsionfree_regular_higher_K-theory} implies all the previous conjectures,
namely Conjectures~\ref{con:K_0(ZG)_vanishes_for_torsionfree_G},~%
\ref{con:K_0(ZG)_vanishes_for_torsionfree_G_and_regular_R},~%
\ref{con:Idempotent_Conjecture},~\ref{con:Serre},~%
\ref{con:FHC_Farrell-Jones_Conjecture_for_Wh(G)_for_torsionfree_Groups},~%
\ref{con:FJC_neg_K-Theory_and_regular_coefficient_rings_and_torsionfree_groups},
and~\ref{con:The_Farrell-Jones_Conjecture_for_middle_and_negative_K-theory}.


\typeout{---------- Section 6:  Algebraic $L$-theory    ---------------}

\section{Algebraic $L$-theory}\label{sec:Algebraic_L-theory}


\subsection{Algebraic $L$-groups}%
\label{subsec:Algebraic_L-groups}

A \emph{ring with involution} $S$ is an
associative ring $S$ with unit together with an \emph{involution of rings}
\[
\overline{\ }\colon  S \to S, \quad s \mapsto \overline{s},
\] 
i.e., a map satisfying $\overline{\overline{s}} = s$, $\overline{s_0 +s_1} =
\overline{s_0} + \overline{s_1}$, $\overline{s_0\cdot s_1} = \overline{s_1}\cdot
\overline{s_0}$, and $\overline{1} = 1$ for $s_0,s_1 \in S$.
If $S$ is commutative, we can equip it with the trivial involution $\overline{s} = s$.

Let $w\colon  G \to \{\pm 1\}$ be a group homomorphism.
Then the group ring $S = RG$ inherits an involution, the so-called \emph{$w$-twisted involution}
that sends $\sum_{g \in G} r_g \cdot g$ to  $\sum_{g \in G} w(g) \cdot \overline{r_g} \cdot g^{-1}$.

Below we fix a ring $S$ with involution. Module is to be understood as left module unless
explicitly stated otherwise. The main purpose of the involution is to ensure that the dual
of a left $S$-module can be viewed as a left $S$-module again.  Namely, let $M$ be a left
$S$-module.  Then $M^* := \hom_S(M,S)$ carries a canonical right $S$-module structure
given by $(fs)(m) = f(m) \cdot s$ for a homomorphism of left $S$-modules $f\colon M \to S$
and $m \in M$.  The involution allows us to view $M^* = \hom_S(M,S)$ as a left $S$-module,
namely, define $sf$ for $s\in S$ and $f \in M^*$ by $(sf)(m) := f(m)\cdot \overline{s}$
for $m \in M$.

Given a finitely generated projective $S$-module $P$, denote by
$e(P) \colon P \xrightarrow{\cong} (P^*)^*$ the canonical isomorphism of left
$S$-modules that sends $p \in P$ to the element in $(P^*)^*$ given by
$P^* \to S, \;f \mapsto \overline{f(p)}$. We will use it in the sequel to identify
$P = (P^*)^*$.

  \begin{definition}[Non-singular $\epsilon$-symmetric
    form]\label{def:symmetric-non-singular-form}
    Let $\epsilon \in \{\pm 1\}$. An \emph{$\epsilon$-symmetric form} over an associative
    ring $S$ with involution is a finitely generated projective $S$-module $P$
    together with an $S$-map $\phi\colon P \to P^*$ such that the composite
    $P \xrightarrow{e(P)} (P^*)^* \xrightarrow{\phi^*} P^*$ agrees with
    $\epsilon \cdot \phi$.

    We call an $\epsilon$-symmetric form $(P,\phi)$ \emph{non-singular}
    if $\phi$ is an isomorphism.
  \end{definition}

\begin{definition}[The standard hyperbolic $\epsilon$-symmetric form]%
\label{def:standard_symmetric_hyperbolic_form}
    Let $P$ be a finitely generated projective $S$-module. The \emph{standard hyperbolic
      $\epsilon$-symmetric form} $H^{\epsilon}(P)$ is given by the $S$-module $P \oplus P^*$ and the
    $S$-isomorphism
    \[
      \phi\colon (P \oplus P^*) \xrightarrow{\begin{pmatrix} 0& 1\\ \epsilon & 0
        \end{pmatrix}}
       P^* \oplus
      P \xrightarrow{\id \oplus e(P)} P^* \oplus (P^*)^* \xrightarrow{\gamma} (P \oplus
      P^*)^*
    \]
    where $\gamma$ is the obvious $S$-isomorphism.
  \end{definition}

Let $P$ be a finitely generated projective $S$-module. Define an involution of abelian groups
\[
T  = T(P)
\colon  \hom_S(P,P^*) \to \hom_S(P,P^*), \quad u \mapsto u^* \circ e(P).
\]
 Define abelian groups
\begin{eqnarray*}
Q^{\epsilon}(P)
 & := & \ker\left((1- \epsilon \cdot T)\colon  \hom_S(P,P^*) \to \hom_S(P,P^*)\right);
\\
Q_{\epsilon}(P)
 & := & \coker\left((1 - \epsilon \cdot T)\colon  \hom_S(P,P^*) \to \hom_S(P,P^*)\right).
\end{eqnarray*}

\begin{definition}[Non-singular $\epsilon$-quadratic form]%
\label{def:quadratic_non-singular_form}
Let $\epsilon \in \{\pm 1\}$. An \emph{$\epsilon$-quadratic form}
$(P,\psi)$
is a finitely generated projective $S$-module $P$ together with an
element $\psi \in Q_{\epsilon}(P)$. It is called
\emph{non-singular} if the \emph{associated $\epsilon$-symmetric form} $(P,(1 + \epsilon \cdot T)(\psi))$
is non-singular, i.e.
$(1 + \epsilon \cdot T)(\psi)\colon  P \to P^*$ is bijective.
\end{definition}

\begin{definition}[The standard hyperbolic $\epsilon$-quadratic form]%
\label{def:standard_quadratic_hyperbolic_form} 
Let $P$ be a finitely generated projective $S$-module. The 
\emph{standard hyperbolic $\epsilon$-quadratic form}
$H_{\epsilon}(P)$ is given by the $S$-module $P \oplus P^*$ and 
the class in $Q_{\epsilon}(P \oplus P^*)$
of the $S$-homomorphism
\[
\phi\colon  (P \oplus P^*) \xrightarrow{\begin{pmatrix} 0 &1 \\ 0 & 0  \end{pmatrix}}
P^* \oplus P \xrightarrow{\id \oplus e(P)}  P^* \oplus (P^*)^* \xrightarrow{\gamma} (P \oplus P^*)^*
\]
where $\gamma$ is the obvious $S$-isomorphism.
\end{definition}

\begin{definition}[$L$-groups in even dimensions]%
\label{def:quadratic_Lh-group_in_even_dimensions}
For an even integer $n = 2k$ define the abelian group $L^p_n(S)$,
called the \emph{$n$-th projective quadratic $L$-group,}
to be the abelian group of
equivalence classes $[P,\psi]$ of non-singular $(-1)^{k}$-quadratic 
forms $(P,\psi)$, whose underlying $S$-module $P$ is
a finitely generated projective  $S$-module, with 
respect to the following equivalence relation: We call
$(P,\psi)$ and $(P',\psi')$ equivalent or stably isomorphic
if and only  if there exists finitely generated projective $S$-modules $Q$ and $Q'$ and an 
isomorphism of non-singular $(-1)^k$-quadratic forms 
\[
(P,\psi) \oplus H_{(-1)^k}(Q) \cong (P',\psi') \oplus H_{(-1)^k}(Q').
\]
Addition is given by the direct sum of two $(-1)^k$-quadratic forms. 
The zero element is represented
by $[H_{(-1)^k}(Q)]$ for any finitely generated projective $S$-module $Q$.  The inverse of
$[P,\psi]$ is given by $[P,-\psi]$. 
\end{definition}

Note that $L_0(S)$ corresponds to $K_0(S)$ in the sense that it can be viewed as the
Grothendieck construction associated to non-singular $(-1)^{k}$-quadratic forms $(P,\psi)$
modulo hyperbolic forms. In view of the definition of $K_1(S)$ in terms of automorphism,
it is not surprizing that one can define for an odd integer $n=2k+1$ the abelian group
$L^p_n(S)$ in terms of automorphism of $(-1)^{k}$-quadratic forms
$(P,\psi)$. For more detail we refer for instance to~\cite[Section~9.2]{Lueck-Macko(2024)}.

If one replaces in the definition of $L_n^p(S)$ finitely generated projective by finitely
generated free, one obtains the \emph{$n$-th free quadratic $L$-group} $L_n^h(S)$.

If $S$ is $\IZ G$ with the involution sending $\sum_{g \in G}r_g \cdot g$ to
$\sum_{g \in G}r_g \cdot g^{-1}$, there is also the \emph{$n$-th  simple quadratic $L$-group} $L_n^s(\IZ G)$, where
one considers based free finitely generated $\IZ G$-modules and takes Whitehead torsion
into account.

Actually, for every $j \in \IZ^{\le  2}$, one can define  $L$-groups   $L^{\langle j\rangle}_n (\IZ G)$,
where we put
\begin{eqnarray*}
  L_n^s(\IZ G) & = & L^{\langle 2 \rangle}_n(\IZ G);
  \\
  L^h_n(\IZ G) & = &   L^{\langle 1 \rangle}_n(\IZ G);
  \\
  L^p_n(\IZ G)
  & = & L^{\langle 0 \rangle}_n(\IZ G).
\end{eqnarray*}
For $j \in \IZ^{\le 1}$ there are canonical maps
$L_n^{\langle j+1 \rangle}(\IZ G) \to L_n^{\langle j\rangle}(\IZ G)$ which fit into
the so called \emph{Rothenberg sequences}
\begin{multline}
  \cdots \to L_n^{\langle j+1 \rangle}(\IZ G) \to L_n^{\langle j\rangle}(\IZ G) \to
  \widehat{H}^n(\IZ/ 2; \widetilde{K}_j(\IZ G))  \\ \to L_{n-1}^{\langle j+1 \rangle}(\IZ G)
    \to L_{n-1}^{\langle j\rangle}(\IZ G) \to \cdots
 \label{Rothenberg_sequence_for_Zpi_j_plus_1_to_j}
\end{multline}
involving the Tate cohomology of the $K$-groups of $\IZ G$.

One defines
\begin{eqnarray}
 L^{\langle - \infty \rangle}(\IZ G)  & = & \colim_{j \to - \infty} L_n^{\langle j\rangle}(\IZ G).
\label{L-groups_with_decorations_-infty}
\end{eqnarray}

For a group homomorphism $w \colon G \to \{\pm 1\}$, there are also versions
$L_n^{\langle -\infty \rangle}(\IZ G,w)$ which are just obtained by the constructions above
applied to $\IZ G$ equipped with the $w$-twisted involution.  In applications to  closed
manifolds, the group $G$ will be given by the fundamental group and $w$ by the first
Stiefel-Whitney class.

Note that by definition we have
$L_n^{\langle j \rangle} (\IZ G) = L_{n+4}^{\langle j \rangle} (\IZ G)$ for every $n \in \IZ$ and
$j \in \{2,1,0,-1, \ldots\} \amalg \{-\infty\}$.

These decorations $s,p,h$, or $\langle j \rangle$ do not appear in $K$-theory and lead to
some complications for $L$-theory. Namely, for geometric applications the decoration $s$
is the right one in view of the $s$-Cobordism Theorem~\ref{the:s-Cobordism_Theorem},
whereas for the Farrell-Jones Conjecture one needs to consider the decoration
$\langle - \infty \rangle$ because of the Shaneson splitting
\begin{eqnarray*}
  L_n^{\langle j \rangle}(\IZ[G \times \IZ])
  & \cong &
L_n^{\langle j \rangle}(\IZ G ) \oplus L_{n-1}^{\langle j-1 \rangle}(\IZ G );
\\
  L_n^{\langle -\infty \rangle}(\IZ[G \times \IZ])
  & \cong &
 L_n^{\langle -\infty  \rangle }(\IZ G) \oplus L_{n-1}^{\langle -\infty \rangle}(\IZ G).
\end{eqnarray*}

However, there is a favourite situation, where the decoration do not matter because of the
Rothenberg sequences.

\begin{proposition}\label{pro:no_decorations_in_the_torsionfree_case}
  Suppose  that  $G$ satisfies
  Conjecture~\ref{con:The_Farrell-Jones_Conjecture_for_middle_and_negative_K-theory}
  Then the canonical map
  \[
    L_n^s(\IZ G) \to L_n^{\langle -\infty  \rangle }(\IZ G)
  \]
  is bijective.
\end{proposition}


\subsection{Conjecture about $L$-groups for  torsionfree groups}%
\label{subsec:Conjecture_about_K-groups_for_torsionfree_groups}

Given a ring with involution $S$, there
exists an \emph{$L$-theory spectrum associated to $S$ with decoration $\langle -\infty \rangle$}
\begin{eqnarray}
& \bfL^{\langle -\infty \rangle}(S)%
\label{spectrum_Llangle-inftyrangle(S)}
\end{eqnarray}
with the property that $\pi_n(\bfL^{\langle -\infty \rangle}(S)) =
L_n^{\langle -\infty \rangle}(S)$ holds for $n \in \IZ$.

\begin{conjecture}[Farrell-Jones Conjecture for $L$-theory for torsionfree groups]%
\label{con:FJC_torsionfree_L-theory}%
Let $G$ be a torsionfree group. Let $R$ be any ring with involution.  

Then the assembly map
\[
  H_n(BG;\bfL^{\langle -\infty \rangle} (R)) \to L_n^{\langle - \infty
    \rangle}(RG)
\]
is an isomorphism for all $n \in \IZ$.
\end{conjecture}

Note that $L_n^{\langle j \rangle}(\IZ)$ is independent of the
decoration and hence we simple write $\bfL$ or $\bfL(\IZ)$ for
$\bfL^{\langle j \rangle} (\IZ)$.

\begin{proposition}\label{FJC_for_torsionfree_G_and_decoration_s}
  Let $G$ be a torsionfree group coming with a group homomorphism
  $w \colon G \to \{\pm 1\}$.  Suppose that
  Conjectures~\ref{con:The_Farrell-Jones_Conjecture_for_middle_and_negative_K-theory}
  and~\ref{con:FJC_torsionfree_L-theory} hold for $R = \IZ$.

  Then the assembly map
  \[
    H_n(BG;\bfL(\IZ)) \to L_n^s(\IZ G)
\]
is an isomorphism for all $n \in \IZ$.
\end{proposition}


\typeout{---------- Section 6:  Further applications of the farrell-Jones Conjecture  for torsionfree groups    ---------------}

\section{Further applications of the Farrell-Jones Conjecture  for torsionfree groups}%
\label{sec:Further_applications_of_the_farrell-Jones_Conjecture__for_torsionfree_groups}

We have already discussed the finiteness obstruction in
Subsection~\ref{subsec:Wall's_finiteness_obstruction} and the $s$-Cobordism Theorem in
Subsection~\ref{subsec:The_s-Cobordism_Theorem} as interesting applications of
Conjectures~\ref{con:K_0(ZG)_vanishes_for_torsionfree_G}
and~\ref{con:FHC_Farrell-Jones_Conjecture_for_Wh(G)_for_torsionfree_Groups} for
torsionfree fundamental groups.  Next we discuss prominent applications of
Conjectures~\ref{con:The_Farrell-Jones_Conjecture_for_middle_and_negative_K-theory}
and~\ref{con:FJC_torsionfree_L-theory} for $R = \IZ$ which illustrates their importance.


\subsection{The Borel Conjecture}%
\label{subsec:The_Borel_Conjecture}

\begin{definition}[Topologically rigid]\label{def:topological_rigid}
   We call a closed topological manifold $N$
  \emph{topologically rigid} if any homotopy equivalence $M \to N$
   with a closed topological manifold $M$ as source is homotopic to a homeomorphism.
 \end{definition}

 The Poincar\'e Conjecture says that $S^n$ is topologically rigid and is known to be true
 for all $n \ge 1$.

 A space $X$ is called \emph{aspherical}, if it is path connected and all its higher
 homotopy groups vanish, i.e., $\pi_n(X)$ is trivial for $n \ge 2$.  A $CW$-complex is
 aspherical if and only if it is connected and its universal covering is contractible.
 Given two aspherical $CW$-complexes $X$ and $Y$, the map from the set of homotopy classes
 of maps $X \to Y$ to the set of group homomorphisms $\pi_1(X) \to \pi_1(Y)$ modulo inner
 automorphisms of $\pi_1(Y)$ given by the map induced on the fundamental groups is a
 bijection.  In particular, two aspherical $CW$-complexes are homotopy equivalent if and
 only if they have isomorphic fundamental groups and every isomorphism between their
 fundamental groups comes from a homotopy equivalence.

 \begin{example}[Examples of aspherical
  manifolds]\label{exa:Examples_of_aspherical_manifolds}\
  \begin{itemize}
  \item A connected closed $1$-dimensional manifold is homeomorphic to $S^1$ and
    hence aspherical;

  \item Let $M$ be a connected closed $2$-dimensional manifold.  Then $M$ is
    either aspherical or homeomorphic to $S^2$ or $\IR\IP^2$;

  \item A connected closed $3$-manifold $M$ is called \emph{prime} if for any
    decomposition as a connected sum $M \cong M_0 \sharp M_1$ one of the
    summands $M_0$ or $M_1$ is homeomorphic to~$S^3$. It is called \emph{irreducible}, if
    any embedded    sphere $S^2$ bounds a disk $D^3$. Every irreducible closed $3$-manifold is
    prime. A prime closed $3$-manifold is either irreducible or an $S^2$-bundle
    over $S^1$.  The following statements are equivalent for a closed
    $3$-manifold $M$:

    \begin{itemize}

    \item $M$ is aspherical;

    \item $M$ is irreducible and its fundamental group is infinite and contains
      no element of order $2$;

    \item The fundamental group $\pi_1(M)$ cannot be written in a non-trivial way as a
      free product of two groups, is infinite, different from $\IZ$, and contains no element of order $2$;

    \item The universal covering of $M$ is homeomorphic to $\IR^3$.

     \end{itemize}

    \item Let $L$ be a Lie group with finitely many path components. Let $K
      \subseteq L$ be a maximal compact subgroup. Let $G \subseteq L$ be a
      discrete torsionfree subgroup.  Then $M = G\backslash L/K$ is an
      aspherical closed smooth manifold with fundamental group $G$, since its universal
      covering $L/K$ is diffeomorphic to $\IR^n$ for appropriate $n$;

    \item Every closed Riemannian smooth manifold with non-positive sectional curvature
      has a universal covering which is diffeomorphic to $\IR^n$ and is in
      particular aspherical.

    \end{itemize}
  \end{example}

\begin{example}[Examples of exotic aspherical closed manifolds]%
\label{exa:Examples_of_exotic_aspherical_closed_manifolds}
  There are rather exotic aspherical closed topological manifolds.
  The following examples are due to Belegradek~\cite[Corollary~5.1]{Belegradek(2006)},
  Davis~\cite{Davis(1983)},  Davis-Fowler-Lafont~\cite{Davis-Fowler-Lafont(2014)},
  Davis-Januszkie\-wicz~\cite[Theorem~5b.1]{Davis-Januszkiewicz(1991)},
  Mess~\cite{Mess(1991)}, Osajda~\cite[Corollary~3.5]{Osajda(2020)}, and
  Weinberger, see~\cite[Section~13]{Davis(2002exotic)}.

  \begin{itemize}

  \item There exists for each $n \ge 6$ an $n$-dimensional aspherical closed topological
    manifold that  cannot be triangulated. One can arrange  that the fundamental group is hyperbolic;

\item For each $n \ge 4$   there exists an aspherical closed  $n$-dimensional
  manifold such that its universal covering is not homeomorphic to $\IR^n$;

\item For every $n \ge 4$, there is an aspherical closed  topological manifold of
    dimension $n$ whose fundamental group contains an infinite
    divisible abelian group;

  \item For every $n \ge 4$, there is an  aspherical closed  $\PL$ manifold of
    dimension $n$ whose fundamental group has an unsolvable word
    problem and whose simplicial volume in the sense of Gromov~\cite{Gromov(1982)} is non-zero;

    \item For every $n \ge 4$, there is an  aspherical closed   manifold of
      dimension $n$  whose fundamental group contains coarsely embedded expanders;

  \end{itemize}
\end{example}

 \begin{conjecture}[Borel Conjecture (for a group $G$ in  dimension $n$)]%
\label{con:Borel_Conjecture}
  The \emph{Borel Conjecture for a group $G$ in  dimension $n$}
   predicts for two aspherical closed topological manifolds $M$ and $N$ of dimensions $n$
  with $\pi_1(M) \cong \pi_1(N) \cong G$, that $M$ and $N$ are
  homeomorphic and any homotopy equivalence $M \to N$ is homotopic to
  a homeomorphism.

  The \emph{Borel Conjecture}
  says that every aspherical closed topological manifold is topologically rigid.  
\end{conjecture}

One of the main applications of the Farrell-Jones Conjecture is the next result.

\begin{theorem}\label{the:FJC_implies_Borel} Suppose that the torsionfree group $G$
  satisfies Conjectures~\ref{con:The_Farrell-Jones_Conjecture_for_middle_and_negative_K-theory}
  and~\ref{con:FJC_torsionfree_L-theory} for $R = \IZ$. Consider $n \in \IZ$ with
  $n \ge 5$.

  Then the Borel Conjecture holds for $G$ in dimension $n$.
\end{theorem}

We sketch the proof of Theorem~\ref{the:FJC_implies_Borel}. More details can be found for instance
in~\cite[Chapter~19]{Lueck-Macko(2024)}. We need some preparations.

\begin{definition}[Simple structure set]%
\label{def:simple_structure_set}
  Let $N$ be a closed topological manifold of dimension $n$. We call two simple homotopy equivalences
  $f_i\colon M_i \to N$ from closed topological manifolds $M_i$ of dimension $n$ to $N$ for
  $i = 0,1$ equivalent if there exists a homeomorphism  $g\colon M_0 \to M_1$ such that
  $f_1 \circ g$ is homotopic to $f_0$.

  The \emph{simple structure set} $\cals_n^s(N)$
 of $N$ is the set of equivalence classes of simple homotopy
  equivalences $M \to N$ from closed manifolds of dimension $n$ to $N$. This set has a
  preferred base point, namely, the class of the identity $\id\colon N \to N$.
\end{definition}

The elementary proof of the next result is left to the reader.

\begin{lemma}\label{lem:topological_rigid_and_the_structure_set}
  Let $N$ be a closed topological manifold with $\Wh(\pi_1(N)) = \{0\}$.
  Then $N$ is topologically rigid if and only if
  $\cals_n^s(N)$ consists of precisely one element.
\end{lemma}

\begin{proof}[Sketch of the proof of Theorem~\ref{the:FJC_implies_Borel}]
Let $G$ be the fundamental  group and $n \ge 5 $ be the dimension
of the aspherical closed topological manifold
$N$.  We assume for simplicity that $N$ is orientable.
Let $\bfL(\IZ)\langle 1 \rangle$ be the $1$-connected cover of $\bfL(\IZ)$.
  This spectrum comes with a map of spectra 
  $\bfi \colon \bfL(\IZ)\langle 1 \rangle \to \bfL(\IZ)$ 
  such  that $\pi_k(\bfi)$ is bijective for $k \ge 1$ and 
  $\pi_k(\bfL(\IZ)\langle 1 \rangle) = 0$ for $k \le 0$. For $k \ge 1$ there is a
  connective simple version of the assembly map
  $\asmb_k^s\colon   H_k\bigl(B\pi;\bfL(\IZ)\bigr) \to  L_k^s(\IZ \pi)$
  appearing in Proposition~\ref{FJC_for_torsionfree_G_and_decoration_s}
  which we denote by
  \[\asmb_k^s\langle 1 \rangle \colon
  H_k\bigl(BG;\bfL(\IZ)\langle 1 \rangle\bigr) \to
  L_k^s(\IZ G).
  \]
  Now there exists an exact sequence of abelian groups coming from the Surgery Exact
  Sequence applied to the aspherical closed manifold $N$
  \begin{multline}
    H_{n+1}\bigl(B\pi;\bfL(\IZ)\langle 1 \rangle\bigr)
    \xrightarrow{\asmb_{n+1}^s\langle 1 \rangle}
    L_{n+1}^s(\IZ \pi) \to \cals_n^s(N)
    \\
    \to H_{n}\bigl(B\pi;\bfL(\IZ)\langle 1 \rangle\bigr)
    \xrightarrow{\asmb_{n}^s\langle 1 \rangle}
    L_n^s(\IZ \pi).
    \label{surgery_exact_sequence}
  \end{multline}
  Since Conjectures~\ref{con:FJC_torsionfree_regular_higher_K-theory}
  and~\ref{con:FJC_torsionfree_L-theory} hold by assumption, we conclude
  from Proposition~\ref{FJC_for_torsionfree_G_and_decoration_s}
  that   that $\asmb_k^s$ is bijective for $k = n,n+1$.
  We have $\asmb_k^s  = \asmb^s_k \circ H_k(\id_{B\pi};\bfi)$. A
  comparison argument of the Atiyah-Hirzebruch spectral sequences shows that 
  $\asmb_{n+1}^s\langle 1 \rangle$ is bijective and in
  particular surjective and that $\asmb_n^s\langle 1 \rangle$
  is injective. We conclude from the sequence~\eqref{surgery_exact_sequence}
  that $\cals_n^s(N)$ consists of one element.
  Lemma~\ref{lem:topological_rigid_and_the_structure_set} implies that
  $N$ is topologically rigid.
\end{proof}

\begin{remark}[The Borel Conjecture in low dimensions]%
  \label{rem:The_Borel_Conjecture_in_low_dimensions}
  The Borel Conjecture is true in dimension $\le 2$ by the classification of closed
  $2$-dimensional manifolds.  It is true in dimension $3$ if Thurston's Geometrization
  Conjecture is true.  This follows from results of Waldhausen, see
  Hempel~\cite[Lemma~10.1 and Corollary~13.7]{Hempel(1976)}, and Turaev,
  see~\cite{Turaev(1988)}, as explained for instance
  in~\cite[Section~5]{Kreck-Lueck(2009nonasph)}.  A proof of Thurston's Geometrization
  Conjecture is given in~\cite{Kleiner-Lott(2008),Morgan-Tian(2014)} following ideas of
  Perelman. In dimension $n = 4$ Theorem~\ref{the:FJC_implies_Borel} is known to be true
  if one assumes that $G$ is good in the sense of Freedman.
\end{remark}

\begin{remark}[The smooth Borel Conjecture holds in dimension $n$ if and only if $n \le 3$]%
\label{rem:The_Borel_Conjecture_does_not_hold_in_the_smooth_category}
The Borel Conjecture~\ref{con:Borel_Conjecture} is false in the smooth category, 
i.e., if
one replaces topological manifold by smooth manifold and homeomorphism by
diffeomorphism. The torus $T^n$ for $ n \ge 5$ is an example,
see~\cite[15A]{Wall(1999)}. The smooth Borel Conjecture is also false in dimension~$4$,
see
Davis-Hayden-Huang-Ruberman-Sunukjian~\cite{Davis-Hayden-Huang-Ruberman-Sunukjian(2024)}.
So the smooth Borel Conjecture is true in dimension $n$ if and only if $n \le 3$, since in
dimension $\le 3$ there is no difference between the smooth and the topological category,
see Moise~\cite{Moise(1952), Moise(1977)},
and the Borel Conjecture~\ref{con:Borel_Conjecture} holds in dimension $\le 3$, see
Remark~\ref{rem:The_Borel_Conjecture_in_low_dimensions}.

\end{remark}

\begin{remark}[The Borel Conjecture versus Mostow rigidity]%
\label{rem:The_Borel_Conjecture_versus_Mostow_rigidity}
  Farrell-Jones~\cite[Theorem~0.1]{Farrell-Jones(1989b)} construct  for given $\epsilon > 0$ a closed
  Riemannian manifold $M_0$ whose sectional curvature lies in the
  interval $[1-\epsilon,-1 + \epsilon]$ and a closed hyperbolic
  manifold $M_1$ such that $M_0$ and $M_1$ are homeomorphic but not
  diffeomorphic.  The idea of the construction is essentially to take
  the connected sum of $M_1$ with exotic spheres.  Note that by
  definition $M_0$ were hyperbolic if we could take $\epsilon = 0$.
  Hence this example is remarkable in view of
  \emph{Mostow rigidity}, 
  which predicts for two closed hyperbolic manifolds $N_0$ and $N_1$
  that they are isometrically diffeomorphic if and only if $\pi_1(N_0)
  \cong \pi_1(N_1)$ and any homotopy equivalence $N_0 \to N_1$ is
  homotopic to an isometric diffeomorphism.

  One may view the Borel Conjecture as the topological version of
  Mostow rigidity.  The conclusion in the Borel Conjecture is weaker,
  one gets only homeomorphisms and not isometric diffeomorphisms, but
  the assumption is also weaker, since there are many more aspherical
  closed topological manifolds than hyperbolic closed manifolds.
\end{remark}


\subsection{Hyperbolic groups with spheres as boundary}%
\label{subsec:Hyperbolic_groups_with_spheres_as_boundary}

If $G$ is the fundamental group of an $n$-dimensional 
closed Riemannian smooth manifold with negative sectional curvature, 
then $G$ is a hyperbolic group in the sense of Gromov, see for
instance~\cite{Bowditch(1991)},~\cite{Bridson-Haefliger(1999)},~\cite{Ghys-Harpe(1990)},
and~\cite{Gromov(1987)}. Moreover, such a group is torsionfree and its
boundary $\partial G$ is homeomorphic to
a sphere.
This leads to the natural question
whether a torsionfree hyperbolic group with a 
sphere as boundary occurs as
fundamental group of an aspherical closed manifold, see Gromov~\cite[page~192]{Gromov(1993)}. 
In high dimensions this question is answered by the following two theorems
taken from Bartels-L\"uck-Weinberger~\cite{Bartels-Lueck-Weinberger(2010)}.
For the notion of and information about the boundary of a hyperbolic group
and its main properties we refer for instance 
to~\cite{Kapovich+Benakli(2002)}. One of the main ingredients in the proof of
Theorem~\ref{the:Hyperbolic_groups_with_spheres_a_boundary}
is the fact that
both the $K$-theoretic and the $L$-theoretic
Farrell-Jones Conjecture hold for hyperbolic groups,
see~\cite{Bartels-Lueck(2012annals)} and~\cite{Bartels-Lueck-Reich(2008hyper)}.

\begin{theorem}[Hyperbolic groups with  spheres as boundary]%
\label{the:Hyperbolic_groups_with_spheres_a_boundary}
Let $G$ be a torsionfree hyperbolic group and let $n$ be an 
integer  $\geq 6$. Then: 

\begin{enumerate}

\item\label{the:Hyperbolic_groups_with_spheres_a_boundary:equivalent_statements}
The following statements are equivalent:
\begin{enumerate}
  \item\label{the:Hyperbolic_groups_with_spheres_a_boundary:equivalent_statements:sphere}
        The boundary $\partial G$ is homeomorphic to $S^{n-1}$;
  \item\label{the:Hyperbolic_groups_with_spheres_a_boundary:equivalent_statements:manifold}
        There is an aspherical closed topological manifold $M$ 
        such that $G \cong \pi_1(M)$, 
        its universal covering $\widetilde{M}$ is 
        homeomorphic to $\IR^n$
        and the compactification of $\widetilde{M}$ by 
        $\partial G$ is homeomorphic to $D^n$;
\end{enumerate}

\item\label{the:Hyperbolic_groups_with_spheres_a_boundary:uniqueness}
The aspherical closed topological manifold $M$ appearing in the assertion above
is unique up to homeomorphism.
\end{enumerate}
\end{theorem}


\subsection{Product decomposition of aspherical closed manifolds}%
\label{subsec:Product_decomposition_of_aspherical_closed_manifolds}

In this subsection we show that, roughly speaking,
an aspherical closed  topological manifold $M$ is a product $M_1 \times M_2$ if and
only if its fundamental group is a product $\pi_1(M) = G_1 \times
G_2$ and that such a decomposition is  unique up to homeomorphism.

\begin{theorem}[Product decompositions of aspherical closed  manifolds]%
  \label{the:Product_decompositions_of_aspherical_closed_manifolds}
  Let $M$ be an aspherical  closed topological manifold of dimension $n$ with
fundamental group $G = \pi_1(M)$. 
Suppose we have a product decomposition
\[
p_1 \times p_2 \colon G \xrightarrow{\cong}  G_1 \times G_2.
\]
Suppose that $G$, $G_1$, and $G_2$  satisfy  Conjectures~\ref{con:FJC_torsionfree_regular_higher_K-theory}
  and~\ref{con:FJC_torsionfree_L-theory} for $R = \IZ$.

Then $G$, $G_1$, and $G_2$ are Poincar\'e duality groups whose
cohomological dimensions satisfy
\[
n = \cd(G) = \cd(G_1) + \cd(G_2).
\]
Suppose additionally:

\begin{itemize}

\item the cohomological dimension $\cd(G_i)$ is different from $3$, $4$, and
      $5$ for $i = 1,2$;
\item $n \not = 4$ or ($n = 4$ and $G$ is good in the sense
      of Freedman).

\end{itemize}

Then:

\begin{enumerate}

\item\label{the:product_decomposition:(1)}
      There are aspherical  closed  topological manifolds $M_1$ and
      $M_2$ together with isomorphisms
      \[
      v_i \colon \pi_1(M_i) \xrightarrow{\cong}  G_i
      \]
      and maps
      \[
      f_i \colon M \to M_i
      \]
      for $i = 1,2$  such that
      \[
      f = f_1 \times f_2 \colon  M \to M_1 \times M_2
      \]
      is a homeomorphism and $v_i \circ  \pi_1(f_i) = p_i$ (up to inner
      automorphisms) for $i = 1,2$;

\item\label{the:product_decomposition:(2)}
      Suppose we have another such choice of aspherical  closed  topological manifolds $M_1'$ and
      $M_2'$ together with isomorphisms
      \[
      v_i' \colon \pi_1(M_i') \xrightarrow{\cong}  G_i
      \]
      and maps
      \[
      f_i' \colon M \to M_i'
      \]
      for $i = 1,2$ such that the map $f' = f_1' \times f_2'$ is a homotopy
      equivalence and  $v_i' \circ \pi_1(f_i') = p_i$ (up to inner
      automorphisms) for $i = 1,2$. Then there are for $i = 1,2$
      homeomorphisms
      $h_i \colon M_i \to M_i'$
      such that $h_i \circ f_i \simeq f_i'$ and $v_i \circ \pi_1(h_i)
      = v_i'$ holds for $i = 1,2$.

\end{enumerate}
\end{theorem}
\begin{proof}
  The case $n \not=3$ is proved in~\cite[Theorem~6.1]{Lueck(2010asph)}. The case $n = 3$
  is done as follows.  We conclude from~\cite[Theorem~11.1 on page~100]{Hempel(1976)} that
  $G_1 \cong \IZ \cong \pi_1(S^1) $ and $G_2$ is the fundamental group $\pi_1(F)$ of a
  closed surface or the other way around.  Now use the fact that the Borel Conjecture is
  true in dimensions $\le 3$.
\end{proof}


\subsection{The stable Cannon Conjecture}%
\label{subsec:The_stable_Cannon_Conjecture}

Tremendous progress in the theory of $3$-manifolds has been made during the last two decades.
A proof of Thurston's  Geometrization Conjecture for orientable $3$-manifolds
is given in~\cite{Kleiner-Lott(2008)},~\cite{Morgan-Tian(2014)}
following ideas of Perelman. The Virtually Fibering Conjecture was settled by
the work of Agol, Liu, Przytycki-Wise, and 
Wise~\cite{Agol(2008),Agol(2013),Liu(2013), Przytycki-Wise(2014), Przytycki-Wise(2018),Wise(2012raggs),Wise(2012hierachy)}.

 However, the following famous conjecture, taken
 from~\cite[Conjecture~5.1]{Cannon-Swenson(1998)}, is still open at the time of writing.

 \begin{conjecture}[Cannon Conjecture]\label{con:Cannon_conjecture}
   Let $G$ be a
  hyperbolic group. Suppose that its boundary is homeomorphic to $S^2$.

  Then $G$ acts properly cocompactly and isometrically on the $3$-dimensional hyperbolic
  space.
\end{conjecture}

In the torsionfree case it boils down to the following conjecture.

\begin{conjecture}[Cannon Conjecture in the torsionfree case]%
\label{con:Cannon_conjecture_in_the_torsionfree_case}
  Let $G$ be a torsionfree hyperbolic group. Suppose that its boundary is homeomorphic to
  $S^2$.

  Then $G$ is the fundamental group of a closed hyperbolic $3$-manifold.
\end{conjecture}

More information about Conjecture~\ref{con:Cannon_conjecture}
and its status can be found for instance
in~\cite[Section~2]{Ferry-Lueck-Weinberger(2019)} and~\cite{Bonk-Kleiner(2005)}.

The following theorem is taken from~\cite[Theorem~2]{Ferry-Lueck-Weinberger(2019)}.  It is
a stable version of the Conjecture~\ref{con:Cannon_conjecture_in_the_torsionfree_case}
above.  Its proof is based on high-dimensional surgery theory, the theory of homology
$\ANR$-manifolds, and the proof of both the $K$-theoretic and the $L$-theoretic
Farrell-Jones Conjecture for hyperbolic groups, see~\cite{Bartels-Lueck(2012annals)}
and~\cite{Bartels-Lueck-Reich(2008hyper)}.

\begin{theorem}[(Stable Cannon Conjecture)]\label{the:stable_Cannon_Conjecture}
  Let $G$ be a hyperbolic $3$-dimensional Poincar\'e duality group.  Let $N$ be any
  smooth, $\PL$, or topological manifold respectively, that  is closed and whose dimension is
  $\ge 2$.

  Then there is a closed smooth, $\PL$,  or topological manifold $M$ and a normal map
  of degree one
  \[
    \xymatrix{TM \oplus \underline{\IR^a} \ar[d]\ar[r]^-{\underline{f}} & \xi \times TN
      \ar[d]
      \\
      M\ar[r]^-f & BG \times N }
  \]
  satisfying

  \begin{enumerate}

  \item\label{the:stable_Cannon_Conjecture:simple} The map $f$ is a simple homotopy
    equivalence;

  \item\label{the:stable_Cannon_Conjecture:universal_covering} Let $\widehat{M} \to M$ be
    the $G$-covering associated to the composite of the isomorphism
    $\pi_1(f) \colon \pi_1(M) \xrightarrow{\cong} G \times \pi_1(N)$ with the projection
    $G \times \pi_1(N) \to G$. Suppose additionally that $N$ is aspherical, 
    $\dim(N) \ge 3$, and $\pi_1(N)$ is a Farrell-Jones group
    in the sense of Definition~\ref{def:Farrell-Jones-group}.

    Then $\widehat{M}$ is homeomorphic to $\IR^3 \times N$. Moreover, there is a compact topological
    manifold $\overline{\widehat{M}}$ whose interior is homeomorphic to $\widehat{M}$ and
    for which there exists a homeomorphism of pairs
    $(\overline{\widehat{M}},\partial \overline{\widehat{M}}) \to (D^3 \times N, S^2
    \times N)$.

  \end{enumerate}
\end{theorem}

If we could choose $N = \pt$ in Theorem~\ref{the:stable_Cannon_Conjecture}, it would imply
Conjecture~\ref{con:Cannon_conjecture_in_the_torsionfree_case}.


\subsection{Automorphisms groups of aspherical closed manifolds}%
\label{subsec:Automorphisms_groups_of_aspherical_closed_manifolds}

We record the following two results  about the homotopy groups of
the automorphism group of an aspherical  closed   manifold.

\begin{theorem}[Homotopy Groups of $\Top(M)$ rationally for closed aspherical $M$]%
\label{the:auto_top}
Let $M$ be an  aspherical  closed
topological manifold with fundamental group $\pi$.  Suppose
that $M$ is smoothable and $\pi$ satisfies  Conjectures~\ref{con:FJC_torsionfree_regular_higher_K-theory}
  and~\ref{con:FJC_torsionfree_L-theory} for $R = \IZ$.

Then for $1 \leq i \leq ( \dim M -7 ) / 3$
one has
\[
\pi_i ( \Top( M ) ) \otimes_{\IZ} \IQ 
              = \left\{ \begin{array}{lll} \zentrum ( \pi ) \otimes_{\IZ} \IQ  & \quad & \mbox{ if } i=1 ;\\
                                                  0  & \quad & \mbox{ if } i > 1.
                         \end{array}    
                \right.
\]
\end{theorem}

In the differentiable case one additionally needs to study involutions on the higher $K$-theory groups.

\begin{theorem}[Homotopy Groups of $\Diff(M)$ rationally for closed aspherical $M$]%
\label{the:auto_diff}
Let $M$ be an aspherical orientable closed smooth manifold  with fundamental group $\pi$.

Then we have for $1 \leq i \leq (\dim M -7)/3 $
\[
\pi_i ( \Diff ( M ) ) \otimes_{\IZ} \IQ =
                \left\{ \begin{array}{lll} \zentrum ( \pi ) \otimes_{\IZ} \IQ  &  & \mbox{ if } i=1 ;\\
     \bigoplus_{j=1}^{\infty} H_{(i +1) - 4j} ( M ; \IQ ) &  & \mbox{ if } i > 1 \mbox{ and } \dim M \mbox{ odd};\\
                                                      0 & & \mbox{ if } i > 1 \mbox{ and } \dim M \mbox{ even}.
                         \end{array}    
                \right.
\]
\end{theorem}

For a proof see for
instance~\cite{Farrell-Hsiang(1978)},~\cite[Section~2]{Farrell-Jones(1990b)},
and~\cite[Lecture~5]{Farrell(2002)}.  For a  survey on automorphisms of manifolds we refer to~\cite{Weiss-Williams(2001)}.

\begin{remark}[Homotopy Groups of $G(M)$ for closed aspherical $M$]%
\label{rem:Homotopy_Groups_of_G(M)_or_closed_aspherical_M}
  Let $M$ be  an aspherical  closed topological manifold  with fundamental group $\pi$.
  Let $G(M)$ be the monoid of self homotopy equivalences
of $M$. Choose $\id_M$ as its base point.  Then there are  isomorphisms
\begin{eqnarray*}
\alpha_0 \colon \pi_0(G(M))  & \xrightarrow{\cong} & \Out(\pi);
  \\
\alpha_1 \colon \pi_1(G(M))   & \xrightarrow{\cong} & \zentrum(\pi);
  \\
  \pi_n(G(M)) & \cong & \{0\} \quad \text{for}\; n \ge 2,
\end{eqnarray*}
where $\alpha_0$ comes from taking the map induced on the fundamental group
and $\alpha_1$ comes from the evaluation map $G(M) \to M, f \mapsto f(x)$ for a base point $x \in M$,
see~\cite[Theorem~III.2]{Gottlieb(1965)}. Define maps $\beta_n$ for $n = 0,1$ to be the composites
\begin{eqnarray}
 & \beta_0 \colon \pi_0(\Top(M)) \to \pi_0(G(M)) \xrightarrow{\alpha_0} \Out(\pi); &
  \label{beta_0}
  \\
  &\beta_1 \colon \pi_1(\Top(M)) \to \pi_1(G(M)) \xrightarrow{\alpha} \zentrum(\pi). &
  \label{beta_1}
\end{eqnarray}
The maps $\beta_0$ and $\beta_1$ are rationally bijective
if the assumptions appearing in Theorem~\ref{the:auto_top} are satisfied.
\end{remark}

\begin{remark}\label{rem:cohomology_of_BTOP(M)}
  Let $M$ be  an aspherical  closed topological manifold
  such that the  assumptions appearing in Theorem~\ref{the:auto_top} are satisfied.
  Then  get also some information about the (co)homology of $\BTop(M)^{\circ}$, where 
$\Top(M)^{\circ}$ denotes the component of the identity of $\Top(M)$.
We get from $\beta_1$ defined in~\eqref{beta_1} and Theorem~\ref{the:auto_top} 
a map 
\[
\BTop( M )^{\circ}  \to K(\zentrum(\pi),2)
\] 
of simply connected spaces inducing isomorphisms on the rationalized homotopy groups
in dimensions $\le ( \dim M -7 ) / 3 +1$
This implies that we get  isomorphisms
\begin{eqnarray*}
  H_n(\BTop(M)^{\circ};\IQ)& \xrightarrow{\cong} &H_n(K(\zentrum(\pi),2);\IQ) ;
   \\                                                
H^n(K(\zentrum(\pi),2);\IQ) & \xrightarrow{\cong} &H^n(\BTop(M)^{\circ};\IQ),
\end{eqnarray*}
for  $n \le ( \dim M -7 ) / 3 +1$.
\end{remark}

\begin{remark}\label{rem:finitely_generated_center_and_TOP(M)_M_closd_aspherical}
  Let $M$ be  an aspherical   closed topological manifold with fundamental group $\pi$
  such that the  assumptions appearing in Theorem~\ref{the:auto_top} are satisfied.

  Then the following assertions are equivalent by
  Theorem~\ref{the:auto_top} and Remark~\ref{rem:cohomology_of_BTOP(M)}:

  \begin{enumerate}
  \item The abelian group  $\zentrum(\pi)$ of $\pi$ is finitely generated;

  \item The $\IQ$-module $\IQ \otimes_{\IZ} \zentrum(\pi)$ is finitely generated;

  \item The $\IQ$-module  $H^2(\BTop(M)^{\circ};\IQ)$ is finitely generated;

  \item The $\IQ$-module $\IQ \otimes_{\IZ} \pi_1 (\Top( M ))$ is finitely generated;

  \item The $\IQ$-module $\IQ \otimes_{\IZ} \pi_1 ( \Top( M ))$ is finitely generated 
    and the $\IQ$-module $\IQ \otimes_{\IZ} \pi_i  ( \Top( M )) \otimes_{\IZ} \IQ$ is trivial
    for  $2 \leq i \leq (\dim M -7)/3 + 1$;

  \item The $\IQ$-module  $H^i(\BTop(M)^{\circ};\IQ)$ is finitely generated for $1 \leq i \leq (\dim M -7)/3 +1$;

  \end{enumerate}

  Recall the open question whether $\zentrum(\pi_1(M))$ is finitely generated for a closed
  aspherical manifold $M$. 

  In this context we mention the result of
  Budney-Gabai~\cite[Theorem~1.3]{Budney-Gabai(2023)} that, for $n \ge 4$ and an
  $n$-dimensional hyperbolic closed manifold, both $\pi_{n-4}(\Top(M))$ and
  $\pi_{n-4}(\Diff(M)^{\circ})$ are not finitely generated. Note  that $\pi_1(M)$ satisfies the
  Farrell-Jones Conjecture by
  Theorem~\ref{the:status_of_the_Full_Farrell-Jones_Conjecture}~%
\ref{the:status_of_the_Full_Farrell-Jones_Conjecture:Classes_of_groups:hyperbolic_groups},
  $\zentrum(\pi_1(M))$ is trivial, and $\pi_i (\Top( M )) \otimes_{\IZ} \IQ$ and
  $H^i(\BTop(M)^{\circ};\IQ)$ vanish for $1 \leq i \leq (\dim M -7)/3 +1$.
\end{remark}

Integral computations of the homotopy groups of automorphisms are much harder.
We mention at least the following result taken
from~\cite[Theorem~1.3]{Enkelmann-Lueck-Malte-Ullmann-Winges(2018)}.

\begin{theorem}[Homotopy groups of $\Top(M)$ for closed aspherical $M$ with hyperbolic fundamental group]%
  \label{the:Homotopy_groups_of_Top(M))_for_closed_aspherical_M_with_hyperbolic_fundamental_group}
  
Let $M$ be a smoothable aspherical closed topological manifold
whose fundamental group $\pi$ is hyperbolic.

Then there is a $\IZ/2$-action on $\bfWh^{\TOP}(B\pi)$ such that we obtain for every $n$ satisfying
$1 \le n \leq \min\{( \dim M -7 ) / 2, (\dim M - 4)/3\}$ isomorphisms
\[
 \pi_n(\Top(M)) \cong \pi_{n+2}\Big( E\IZ/2_+ \wedge_{\IZ/2} \big( \bigvee_C \bfWh^{\TOP}(BC) \big) \Big)
\]
and an exact sequence
\[
  1 \to \pi_{2}\Big( E\IZ/2_+ \wedge_{\IZ/2} \big( \bigvee_{(C)} \bfWh^{\TOP}(BC) \big) \Big)
  \to\pi_0(\Top(M)) \to \Out(\pi) \to 1
\]
where $(C)$ ranges over the conjugacy classes $(C)$ of maximal infinite cyclic subgroups $C$ of $\pi$.
\end{theorem}

It is conceivable that one can replace the range $(\dim M - 7)/3$ appearing in
Theorem~\ref{the:auto_top}, Theorem~\ref{the:auto_diff}, and
Remark~\ref{rem:finitely_generated_center_and_TOP(M)_M_closd_aspherical} with the range
$\min\{( \dim M -7 ) / 2, (\dim M - 4)/3\}$ appearing in
Theorem~\ref{the:Homotopy_groups_of_Top(M))_for_closed_aspherical_M_with_hyperbolic_fundamental_group}.


\subsection{Poincar\'e duality groups}%
\label{subsec:Poincare_duality_groups}

Next we deal with question which groups occur as fundamental groups of aspherical closed topological
manifolds.

The next  definition is due to Johnson-Wall~\cite{Johnson+Wall(1972)}.
\begin{definition}[Poincar\'e duality group]%
\label{def:Poincare-duality_group}

A group $G$ is called a \emph{Poincar\'e duality group of dimension $n$}
if the following conditions hold:
\begin{enumerate}
\item The group $G$ is of type $\FP$,
i.e., the trivial $\IZ G$-module $\IZ$ 
possesses a finite dimensional
projective $\IZ G$-resolution by finitely generated projective $\IZ G$-modules;

\item We get an isomorphism of abelian groups
\[
H^i(G;\IZ G) \cong 
 \left\{
 \begin{array}{ll}
 \{0\} & \text{for}\;  i \not= n;
 \\
 \IZ & \text{for}\;  i = n.
\end{array} \right.
\]
\end{enumerate}
\end{definition}

A metric space $X$ is called an \emph{absolute neighborhood retract (= $\ANR$)}
if, for every embedding $i \colon X \to Y$ as a closed subspace into a
metric space $Y$, there is an open neighbourhood $U$ of
$\im(i)$ together with a retraction $r \colon U \to \im(i)$, or,
equivalently, for every metric  space $Z$, every closed subset $Y \subseteq Z$, and 
every (continuous) map $f \colon Y \to X$, there exists 
an open neighborhood $U$ of $Y$ in $Z$ together with an 
extension $F \colon U \to X$ of $f$ to $U$.
Every finite $CW$-complex and every  topological manifold is an $\ANR$.

A \emph{finite $n$-dimensional Poincar\'e complex} is a connected finite $CW$-complex $X$
of dimension $n$ and fundamental group $\pi = \pi_1(X)$ together with a  so-called \emph{orientation
homomorphism} $w \colon \pi \to \{\pm 1\}$ and a class
$[X] \in H_n(\IZ^w \otimes_{\IZ \pi} C_*(\widetilde{X}))$ for $\IZ^w$ the $\IZ \pi$-module
whose underlying abelian group is $\IZ$ and on which $x \in \pi$ acts by multiplication
with $w(x)$ such that the up to $\IZ \pi$-chain homotopy unique $\IZ \pi$-chain map
$C^{n-*}(\widetilde{X}) \to C_*(\widetilde{X})$ given by the cap-product with a cycle representing $[X]$
is a $\IZ \pi$-chain homotopy equivalence.
Here the dual chain complex $C^{n-*}(\widetilde{X})$ is to be understood with respect to
the $w$-twisted involution $\IZ \pi \to \IZ \pi$ sending $\sum_{x \in \pi} a_x \cdot x$ to
$\sum_{x \in \pi} a_x \cdot w(x) \cdot x^{-1}$.  Note that this implies that
$H_n(\IZ^w \otimes_{\IZ \pi} C_*(\widetilde{X}))$ is infinite cyclic and $[X]$ is a generator.
This definition is due to Wall~\cite{Wall(1967)}.

A \emph{compact $n$-dimensional  homology $\ANR$-manifold $X$} is
a compact absolute neighborhood retract
such that it has a countable basis for its topology, has finite topological dimension,
and for every $x \in X$ the abelian group $H_i(X,X\setminus \{x\})$ is trivial
for $i \not= n$ and infinite cyclic for $i = n$. A closed $n$-dimensional
topological manifold is an example of a compact
$n$-dimensional homology $\ANR$-manifold,
see~\cite[Corollary 1A in V.26 page~191]{Daverman(1986)}.

\begin{theorem}[Homology $\ANR$-manifolds and finite Poincar\'e complexes]%
\label{the:manifolds_and_Poincare}
Let $M$ be a closed topological manifold, or more generally, a compact
homology $\ANR$-manifold of dimension $n$. Then $M$ is homotopy
equivalent to a finite $n$-dimensional Poincar\'e complex.
\end{theorem}
\begin{proof}
A closed topological manifold, and more generally a compact $\ANR$, has the homotopy 
type of a finite $CW$-complex, see~\cite[Theorem~2.2]{Kirby-Siebenmann(1977)},~\cite{West(1977)}.
The usual proof of Poincar\'e duality for closed manifolds carries over to homology $\ANR$-manifolds.
\end{proof}

\begin{theorem}[Poincar\'e duality groups]%
\label{the:Poincare_duality_groups_versus_Poincare_complexes}
Let $G$ be a group and $n \ge 1$ be an integer. Then:

\begin{enumerate}

\item\label{the:Poincare_duality_groups_versus_Poincare_complexes:G_to_BG}
The following assertions are equivalent:
\begin{enumerate}

\item\label{the:Poincare_duality_groups_versus_Poincare_complexes:G_to_BG:(1)}
$G$ is finitely presented and a Poincar\'e duality group of dimension $n$;

\item\label{the:Poincare_duality_groups_versus_Poincare_complexes:G_to_BG:(2)}
There exists a finitely dominated $n$-dimensional aspherical Poincar\'e complex with $G$ as 
fundamental group;

\end{enumerate}

\item\label{the:Poincare_duality_groups_versus_Poincare_complexes:G_to_BG_K_0(ZG)}
Suppose that $\widetilde{K}_0(\IZ G) = 0$. Then the following assertions are equivalent:
\begin{enumerate}

\item\label{the:Poincare_duality_groups_versus_Poincare_complexes:G_to_BG_K_0(ZG):(1)}
$G$ is finitely presented and a Poincar\'e duality group of dimension $n$;

\item\label{the:Poincare_duality_groups_versus_Poincare_complexes:G_to_BG_K_0(ZG):(2)}
There exists a finite $n$-dimensional aspherical Poincar\'e complex with $G$ as 
fundamental group;

\end{enumerate}

\item\label{the:Poincare_duality_groups_versus_Poincare_complexes:dim_1}
A group $G$ is a Poincar\'e duality group of dimension $1$ if and only if $G \cong \IZ$;

\item\label{the:Poincare_duality_groups_versus_Poincare_complexes:dim_2}
A group $G$ is a Poincar\'e duality group of dimension $2$ if and only if $G$ is isomorphic
to the fundamental group of an aspherical closed surface;

\end{enumerate}
\end{theorem}

Next we state two  very optimistic conjectures, see also
Remark~\ref{rem:homology_ANR-manifolds_versus_topological_manifolds}
and Question~\ref{que:Vanishing_of_the_resolution_obstruction_in_the_aspherical_case}.

\begin{conjecture}[Manifold structures on aspherical Poincar\'e complexes]%
\label{con:Manifold_structures_on_aspherical_Poincare-Complexes}%
Every finitely dominated aspherical Poincar\'e complex
  is homotopy equivalent to a closed topological manifold.
\end{conjecture}
  
\begin{conjecture}[Poincar\'e duality groups]\label{con:Poincare-duality-groups}
A finitely presented group is an $n$-dimensional Poincar\'e duality group if and only if
it is the fundamental group of an  aspherical closed  $n$-dimensional topological manifold.
\end{conjecture}

The \emph{disjoint disk property} says that for any $\epsilon > 0$ and
maps $f,g \colon D^2\to M$ there are maps $f',g' \colon D^2 \to M$ so
that the distance between $f$ and $f'$ and the distance between $g$
and $g'$ are bounded by $\epsilon$ and $f'(D^2) \cap g'(D^2) =
\emptyset$.

\begin{theorem}[Poincar\'e duality groups and aspherical compact homology \linebreak
  $\ANR$-manifolds]%
  \label{the:Borel_FJC_Poincare_complexes}%
  Suppose that  $G$ is a finitely presented Poincar\'e duality group of dimension
  $n \ge 6 $ and satisfies
  Conjectures~\ref{con:The_Farrell-Jones_Conjecture_for_middle_and_negative_K-theory}
  and~\ref{con:FJC_torsionfree_L-theory} for $R = \IZ$.  Let $X$ be some aspherical finite
  Poincar\'e complex with $\pi_1(X) \cong G$. (It exists because of
  Theorem~\ref{the:Poincare_duality_groups_versus_Poincare_complexes}~%
\ref{the:Poincare_duality_groups_versus_Poincare_complexes:G_to_BG_K_0(ZG)}.)  Suppose
  that the Spivak normal fibration of $X$ admits a $\TOP$-reduction.

  Then $BG$ is homotopy equivalent to an aspherical
  compact homology  $\ANR$-manifold
satisfying the disjoint disk property.
\end{theorem}
\begin{proof}
See~\cite[Remark~25.13 on page~297]{Ranicki(1992)},~%
\cite[Main Theorem on page~439 and Section~8]{Bryant-Ferry-Mio-Weinberger(1996)}
and~\cite[Theorem A and Theorem B]{Bryant-Ferry-Mio-Weinberger(2007)}.
\end{proof}

\begin{remark}\label{rem:Erratum} Note that in
  Theorem~\ref{the:Borel_FJC_Poincare_complexes} the condition appears that for some
  aspherical finite Poincar\'e complex $X$ with $\pi_1(X) \cong G$ the Spivak normal
  fibration of $X$ admits a $\TOP$-reduction.  This condition does not appear in earlier
  versions. The reason is that there seems to be a mistake
  in~\cite{Bryant-Ferry-Mio-Weinberger(1996)} as explained in the
  Erratum~\cite{Bryant-Ferry-Mio-Weinberger-Erratum(2024)}. The problem was pointed out by
  Hebestreit-Land-Weiss-Winges, see~\cite{Hebestreit-Land-Weiss-Winges(2024)}. The problem
  is that the proof that any compact homology $\ANR$-manifold has a $\TOP$-reduction of
  its Spivak normal fibration is not correct. For applications
  of~\cite{Bryant-Ferry-Mio-Weinberger(1996)} one has either to assume that the
  $\TOP$-reduction exists or to prove its existence.  This is the reason why this extra
  assumption in Theorem~\ref{the:Borel_FJC_Poincare_complexes} appears.

  As pointed out in~\cite{Bryant-Ferry-Mio-Weinberger-Erratum(2024)},
  Theorems~\ref{the:Hyperbolic_groups_with_spheres_a_boundary}
  and~\ref{the:stable_Cannon_Conjecture} remain true with adding any further
  hypothesis. This is also true for
  Theorem~\ref{the:Product_decompositions_of_aspherical_closed_manifolds}
  as explained in~\cite[Remark~9.185]{Lueck(2022book)}. 
\end{remark}

\begin{remark}[Compact homology $\ANR$-manifolds versus closed topological manifolds]%
\label{rem:homology_ANR-manifolds_versus_topological_manifolds}
 In the following all manifolds have dimension $\ge 6$.
 One would prefer that in the conclusion of
 Theorem~\ref{the:Borel_FJC_Poincare_complexes} one could replace ``compact
 homology $\ANR$-manifold'' by ``closed topological manifold''.  The problem is
 that in the geometric Surgery Exact Sequence  one has to work with the
 $1$-connective cover $\bfL\langle 1 \rangle$ of the $L$-theory
 spectrum $\bfL = \bfL(\IZ) = \bfL(\IZ)^{\langle -\infty \rangle}$,
 whereas in the assembly map appearing in the Farrell-Jones setting one
 uses the $L$-theory spectrum $\bfL$. The $L$-theory spectrum $\bfL$
 is $4$-periodic, i.e., $\pi_n(\bfL) \cong \pi_{n+4}(\bfL)$ for $n \in
 \IZ$. The $1$-connective cover $\bfL\langle 1 \rangle$ comes with a
 map of spectra $\bff\colon \bfL\langle 1 \rangle \to \bfL$ such
 that $\pi_n(\bff)$ is an isomorphism for $n \ge 1$ and
 $\pi_n(\bfL\langle 1 \rangle) = 0$ for $n \le 0$. Since $\pi_0(\bfL)
 \cong \IZ$, one misses a part involving $L_0(\IZ)$ 
 of the so-called \emph{total surgery obstruction} due to Ranicki, i.e., the
 obstruction for a finite Poincar\'e
 complex to be homotopy equivalent to a closed topological manifold. If
 one deals with the periodic $L$-theory spectrum $\bfL$, one picks up only
 the obstruction for a finite Poincar\'e complex to be homotopy
 equivalent to a compact homology $\ANR$-manifold, the so-called \emph{four-periodic
 total surgery obstruction}. The difference of these two obstructions
 is related to the \emph{resolution obstruction} of Quinn, which takes values in $L_0(\IZ)$. 
 Any element of $L_0(\IZ)$ can be realized by an appropriate
 compact homology $\ANR$-manifold as its \emph{resolution obstruction}. There are 
 compact homology $\ANR$-manifolds, that are not homotopy equivalent to
 closed manifolds. But no example of an aspherical compact homology $\ANR$-manifold
 that is not homotopy equivalent to a closed topological manifold is known.
 For  an aspherical compact homology $\ANR$-manifold $M$, the total surgery obstruction
 and the resolution obstruction carry the same information. So we could replace
 in the conclusion of
 Theorem~\ref{the:Borel_FJC_Poincare_complexes}  ``compact
 homology $\ANR$-manifold'' by ``closed topological manifold'' if and only if every aspherical
 compact homology $\ANR$-manifold with the disjoint disk property admits a resolution.

 We refer for instance 
 to~\cite{Bryant-Ferry-Mio-Weinberger(1996),Ferry-Pedersen(1995a),Quinn(1983a),Quinn(1987_resolution),Ranicki(1992)} 
for more information about this topic.
\end{remark}

\begin{question}[Vanishing of the resolution obstruction in the aspherical case]%
\label{que:Vanishing_of_the_resolution_obstruction_in_the_aspherical_case}
Is every aspherical compact homology $\ANR$-manifold homotopy equivalent
to a closed topological manifold?
\end{question}


\typeout{---------- Section 8:  The Full Farrell-Jones Conjecture    ---------------}

\section{The Full Farrell-Jones Conjecture}\label{sec:The_Full_Farrell-Jones_Conjecture}

In this section we formulate the most general version of the Farrell-Jones Conjecture, the
so called \emph{Full Farrell-Jones Conjecture}. It is not of importance to understand the rather
sophisticated formulation of it which we will briefly sketch next but to know that it
implies all the other variants of the Farrell-Jones Conjecture and hence all the
applications explained above and later. So the reader may pass directly to the
Section~\ref{subsec:Farrell-Jones-groups} where we give information about  the class $\calfj$
of Farrell-Jones groups, i.e., groups, which do satisfy the Full Farrell-Jones
Conjecture. It will turn out that $\calfj$ is rather large.


\subsection{Classifying spaces for families of subgroups}\label{subsec:Classifying_spaces_for_families_of_subgroups}

A \emph{family of subgroups of a group $G$} is a  set of subgroups of $G$ that is
closed under conjugation with elements of $G$ and under passing to subgroups.
The most relevant examples for us is  the family $\calvcyc $ of virtually cyclic subgroups,
where virtually cyclic means  that the subgroup is either  finite or contains a normal infinite cyclic subgroup of finite index.
For the notion of a $G$-$CW$-complex for refer for instance to~\cite[Chapters~1 and~2]{Lueck(1989)}.

\begin{definition}[Classifying $G$-$CW$-complex for a family of subgroups]%
\label{def:Classifying_G-CW-for_a_family_of_subgroups}
Let $\calf$ be a family of subgroups of $G$. A model
$\EGF{G}{\calf}$ for the \emph{classifying $G$-$CW$-complex for the family $\calf$ of subgroups of $G$}
or sometimes also called \emph{classifying space for the family $\calf$ of subgroups of $G$}, 
is a $G$-$CW$-complex $\EGF{G}{\calf}$ that has the following properties:
\begin{enumerate}
\item All isotropy groups of $\EGF{G}{\calf}$ belong to $\calf$;
\item For any $G$-$CW$-complex $Y$ whose isotropy groups belong to $\calf$,
there is up to $G$-homotopy precisely one $G$-map $Y \to X$. 
\end{enumerate}
\end{definition}

\begin{theorem} [Homotopy characterization of $\EGF{G}{\calf}$]%
\label{the:G-homotopy_characterization_of_EGF(G)(calf)}%
\index{Theorem!Homotopy characterization of $\EGF{G}{\calf}$}
Let $\calf$ be a family of subgroups.
\begin{enumerate}

\item\label{the:G-homotopy_characterization_of_EGF(G)(calf):existence}
There exists a model for $\EGF{G}{\calf}$ for any family $\calf$;

\item\label{the:G-homotopy_characterization_of_EGF(G)(calf):characterization}
A $G$-$CW$-complex $X$ is a model for $\EGF{G}{\calf}$ if and only if 
all its isotropy groups belong to $\calf$ and for each $H \in\calf$ the $H$-fixed point
set $X^H$ is contractible.
\end{enumerate}
\end{theorem}

For more information about classifying spaces for families we refer for instance to~\cite{Lueck(2005s)}.


\subsection{$G$-homology theories}\label{subsec:G-homology_theories}

\begin{definition}[$G$-homology theory]%
\label{def:G-homology_theory}%
A \emph{$G$-homology theory} $\calh_*^G$ is a collection of
covariant functors 
$\calh^G_n$
from the category of $G$-$CW$-pairs to the category of
abelian groups indexed by $n \in \IZ$ together with natural transformations
\[
\partial_n^G(X,A)\colon  \calh_n^G(X,A) \to
\calh_{n-1}^G(A):= \calh_{n-1}^G(A,\emptyset)
\]
for $n \in \IZ$
such that the following axioms are satisfied:
\begin{itemize}

\item $G$-homotopy invariance \\[1mm]
If $f_0$ and $f_1$ are $G$-homotopic $G$-maps of $G$-$CW$-pairs 
$(X,A) \to (Y,B)$, then $\calh_n^G(f_0) = \calh^G_n(f_1)$ for $n \in \IZ$;

\item Long exact sequence of a pair\\[1mm]
Given a pair $(X,A)$ of $G$-$CW$-complexes,
there is a long exact sequence
\begin{multline*}
\ldots \xrightarrow{\calh^G_{n+1}(j)}
\calh_{n+1}^G(X,A) \xrightarrow{\partial_{n+1}^G(X,A)}
\calh_n^G(A) \xrightarrow{\calh^G_{n}(i)} \calh_n^G(X)
\\
\xrightarrow{\calh^G_{n}(j)} \calh_n^G(X,A)
\xrightarrow{\partial_n^G} \ldots
\end{multline*}
where $i\colon  A \to X$ and $j\colon X \to (X,A)$ are the inclusions;

\item Excision \\[1mm]
Let $(X,A)$ be a $G$-$CW$-pair, and let
$f\colon  A \to B$ be a cellular $G$-map of
$G$-$CW$-complexes. Equip $(X\cup_f B,B)$ with the induced structure
of a $G$-$CW$-pair. Then the canonical map
$(F,f)\colon  (X,A) \to (X\cup_f B,B)$ induces an isomorphism
\[
\calh_n^G(F,f)\colon  \calh_n^G(X,A) \xrightarrow{\cong}
\calh_n^G(X\cup_f B,B)
\]
for all  $n \in \IZ$;

\item Disjoint union axiom\\[1mm]
Let $\{X_i \mid i \in I\}$ be a collection of
$G$-$CW$-complexes. Denote by
$j_i\colon  X_i \to \coprod_{i \in I} X_i$ the canonical inclusion.
Then the map
\[
\bigoplus_{i \in I} \calh^G_{n}(j_i)\colon  \bigoplus_{i \in I} \calh_n^G(X_i)
\xrightarrow{\cong} \calh_n^G\Bigl(\coprod_{i \in I} X_i\Bigr)
\]
is bijective for all $n \in \IZ$;

\end{itemize}
\end{definition}

Let $\cala$ be an additive $G$-category, i.e., an additive category with a $G$-action by
automorphisms of additive $G$-categories. Then one can construct a specific $G$-homology
theory $\calh^G_*(-;\bfK_{\cala})$ such that for a subgroup
$H \subseteq G$ we have
\begin{eqnarray}
  H_n^G(G/H;\bfK_{\cala})& = & K_n(\cala[H])
\label{H_n_upper_G_n(G/H,bfK_cala)_is_K_n(cala(H))}
\end{eqnarray}
where $\cala[H]$ is a certain additive category, which is a kind of twisted group additive
category associated to $H$ and $\cala$ considered as additive $H$-category by restricting the $G$-action to an $H$-action.

If $\cala$ comes additionally with involution compatible with the $G$-action, then
$\cala[G]$ is an additive category with involution and  one can construct a specific $G$-homology
theory $\calh^G_*(-;\bfL_{\cala})$ digesting $G$-$CW$-complexes such that for a subgroup
$H \subseteq G$ we have
\begin{eqnarray}
  H_n^G(G/H;\bfL_{\cala}^{\langle - \infty \rangle}) & = & L_n^{\langle - \infty \rangle}(\cala[H]).
\label{H_n_upper_G_n(G/H,bfL_cala_upper_(langle_minus_infty_rangle))_is_L_n_upper_(langle_minus_infty_rangle)(cala(H))}
\end{eqnarray}

\begin{remark}\label{rem:taking_cala_to_be_calp(R)}
  If $R$ is a ring coming with a group homomorphism $\rho \colon G \to \aut(R)$ to the
  group of ring automorphisms of $R$, then we can take for $\cala$ the category $\calp(R)$
  of finitely generated projective $R$-modules (or rather a skeleton for it) and $\rho$
  induces the structure of an additive $G$-category on $\calp(R)$. Moreover, $\calp(R)[G]$
  can be identified with the category of finitely generated projective modules over
  the twisted group ring $R_{\rho}[G]$.  In particular we get
  \[
    H_n^G(G/H;\bfK_{\calp(R)}) = K_n(R_{\rho|_H}[H])
  \]
  for any subgroup $H \subseteq G$ and
  $n \in \IZ$.

  If $R$ is a ring with involution, the twisted group ring $R_{\rho}[G]$ and
  $\calp(R)[G]$ inherits involutions and we get
  \[
    H_n^G(G/H;\bfL^{\langle - \infty \rangle}_{\calp(R)}) = L_n^{\langle - \infty \rangle}(R_{\rho|_H}[H]).
  \]
\end{remark}

More details of these notions and constructions can be found for instance in~\cite[Chapter~12]{Lueck(2022book)}.


\subsection{The $K$-theoretic version of the Farrell-Jones Conjecture with coefficients in additive categories}%
\label{subsec:The_K-theoretic_version_of_the_Farrell-Jones_Conjecture_with_coeffciients_in_additive_categrories}

\begin{conjecture}[$K$-theoretic Farrell-Jones Conjecture with coefficients in additive $G$-categories]%
\label{con:K-theoretic_Farrell-Jones_Conjecture_with_coefficients_in_additive_G-categories}
We say that $G$ satisfies the \emph{$K$-theoretic Farrell-Jones Conjecture with coefficients in
additive $G$-categories} if for every  additive $G$-category $\cala$ and every $n \in \IZ$ the assembly map
induced  by the projection $\pr \colon \EGF{G}{\calvcyc} \to G/G$
\[
H_n^G(\pr;\bfK_{\cala})  \colon H^G_n(\EGF{G}{\calvcyc};\bfK_{\cala}) 
\to H_n^G(G/G;\bfK_{\cala}) =  K_n(\cala[G])
\]
is bijective.
\end{conjecture}


\subsection{The $L$-theoretic version of the Farrell-Jones Conjecture with coefficients in additive categories}%
\label{subsec:The_L-theoretic_version_of_the_Farrell-Jones_Conjecture_with_coeffciients_in_additive_categrories}

\begin{conjecture}[$L$-theoretic Farrell-Jones Conjecture with coefficients in additive
  $G$-categories with involution]%
\label{con:L-theoretic_Farrell-Jones_Conjecture_with_coefficients_in_additive_G-categories_with_involution}
  We say that $G$ satisfies the \emph{$L$-theoretic Farrell-Jones Conjecture with coefficients
    in additive $G$-categories with involution} if for every  additive $G$-category with involution $\cala$
  and every $n \in \IZ$ the assembly map induced  by the   projection $\pr \colon \EGF{G}{\calvcyc} \to G/G$
\[
  H^G_n(\pr;\bfL_{\cala}^{\langle -\infty \rangle}) \colon H^G_n(\EGF{G}{\calvcyc};\bfL_{\cala}^{\langle -\infty \rangle})
  \to
  H_n^G(G/G;\bfL_{\cala}^{\langle -\infty \rangle}) = L_n^{\langle -\infty \rangle}(\cala[G])
\]
is bijective.
\end{conjecture}


\subsection{The $K$-theoretic version of the Farrell-Jones Conjecture with coefficients in higher  categories}%
\label{subsec:The_K-theoretic_version_of_the_Farrell-Jones_Conjecture_with_coeffciients_in_higher_categories}

\begin{conjecture}[$K$-theoretic Farrell-Jones Conjecture with coefficients in higher $G$-categories]%
\label{con:K-theoretic_Farrell-Jones_Conjecture_with_coefficients_in_higher_G-categories}
We say that $G$ satisfies the \emph{$K$-theoretic Farrell-Jones Conjecture with coefficients in
  higher $G$-categories} if for every right exact $G$-$\infty$-category $\calc$ and every  $n \in \IZ$ the assembly map
 induced  by the   projection $\pr \colon \EGF{G}{\calvcyc} \to G/G$
\[
H_n^G(\pr;\bfK_{\calc})  \colon H^G_n(\EGF{G}{\calvcyc};\bfK_{\calc}) 
\to H_n^G(G/G;\bfK_{\calc}) =  K_n(\calc[G])
\]
is bijective.
\end{conjecture}

Note  that Conjecture~\ref{con:K-theoretic_Farrell-Jones_Conjecture_with_coefficients_in_additive_G-categories}
is a special case of Conjecture~\ref{con:K-theoretic_Farrell-Jones_Conjecture_with_coefficients_in_higher_G-categories}.
It is very likely that one can formulate a version of Conjecture~\ref{con:K-theoretic_Farrell-Jones_Conjecture_with_coefficients_in_higher_G-categories} for $L$-theory which encompasses Conjecture~\ref{con:L-theoretic_Farrell-Jones_Conjecture_with_coefficients_in_additive_G-categories_with_involution} and holds for all Farrell-Jones groups.


\subsection{The formulation of the Full Farrell-Jones Conjecture}%
\label{subsec:The_formulation_of_the_Full_Farrell-Jones_Conjecture}

Let $G$ and $F$ be groups. Their \emph{wreath product} $G \wr F$
is defined as the semidirect product $(\prod_F G) \rtimes F$ where $F$ acts on $\prod_F G$
by permuting the factors.

\begin{conjecture}[Full Farrell-Jones Conjecture]%
\label{con:The_Full_Farrell-Jones_Conjecture}
We say that a group $G$ satisfies the \emph{Full Farrell-Jones Conjecture}
if for any finite group $F$ the wreath product $G \wr F$ satisfies
Conjectures~\ref{con:K-theoretic_Farrell-Jones_Conjecture_with_coefficients_in_additive_G-categories},~%
\ref{con:L-theoretic_Farrell-Jones_Conjecture_with_coefficients_in_additive_G-categories_with_involution},
and~\ref{con:K-theoretic_Farrell-Jones_Conjecture_with_coefficients_in_higher_G-categories}.
\end{conjecture}

For more information about the formulations of the
Conjectures~\ref{con:K-theoretic_Farrell-Jones_Conjecture_with_coefficients_in_additive_G-categories},~%
\ref{con:L-theoretic_Farrell-Jones_Conjecture_with_coefficients_in_additive_G-categories_with_involution},~%
\ref{con:K-theoretic_Farrell-Jones_Conjecture_with_coefficients_in_higher_G-categories},
and~\ref{con:The_Full_Farrell-Jones_Conjecture} we refer for instance to~\cite[Chapter~13]{Lueck(2022book)},
where one can find also references to the relevant literature.


\subsection{The Full Farrell-Jones Conjecture implies all other variants of the Farrell-Jones Conjecture}%
\label{subsec:The_Full_Farrell-Jones_Conjecture_implies_all_other_variants_of_the_Farrell-Jones_Conjecture}

As we have mentioned already above, the Full Farrell-Jones Conjecture~\ref{con:The_Full_Farrell-Jones_Conjecture}
implies all the variants of the Farrell-Jones Conjecture stated in this article, e.g.,
Conjectures~\ref{con:K_0(ZG)_vanishes_for_torsionfree_G},~%
\ref{con:K_0(ZG)_vanishes_for_torsionfree_G_and_regular_R},%
~\ref{con:Idempotent_Conjecture},~\ref{con:Serre},~%
\ref{con:FHC_Farrell-Jones_Conjecture_for_Wh(G)_for_torsionfree_Groups},%
~\ref{con:FJC_neg_K-Theory_and_regular_coefficient_rings_and_torsionfree_groups},%
\ref{con:FJC_torsionfree_regular_higher_K-theory}~%
\ref{con:FJC_torsionfree_L-theory},~%
\ref{con:K-theoretic_Farrell-Jones_Conjecture_with_coefficients_in_additive_G-categories},~%
\ref{con:L-theoretic_Farrell-Jones_Conjecture_with_coefficients_in_additive_G-categories_with_involution},
and~\ref{con:K-theoretic_Farrell-Jones_Conjecture_with_coefficients_in_higher_G-categories}.
This holds also for other version scattered in the literature, for instance for fibered versions and so on.
Proofs or references to the proofs of these implications can be found in~\cite[Section~13.11]{Lueck(2022book)}.


\subsection{Farrell-Jones groups}\label{subsec:Farrell-Jones-groups}

\begin{definition}[Farrell-Jones group]\label{def:Farrell-Jones-group}
  We call a group $G$ a \emph{Farrell-Jones group} if it satisfies the
  Full Farrell-Jones Conjecture~\ref{con:The_Full_Farrell-Jones_Conjecture}.
We denote by $\calfj$ the class of Farrell-Jones groups.
\end{definition}

Proofs or references to the proofs of assertions appearing in the next theorem
can be found in~\cite[Section~16.2]{Lueck(2022book)}.

\begin{theorem}[Status of the  Full Farrell-Jones Conjecture~\ref{con:The_Full_Farrell-Jones_Conjecture}]%
\label{the:status_of_the_Full_Farrell-Jones_Conjecture}\

\begin{enumerate}

\item\label{the:status_of_the_Full_Farrell-Jones_Conjecture:Classes_of_groups}
The following classes of  groups belong to $\calfj$:
\begin{enumerate}

\item\label{the:status_of_the_Full_Farrell-Jones_Conjecture:Classes_of_groups:hyperbolic_groups}
Hyperbolic groups;

\item\label{the:status_of_the_Full_Farrell-Jones_Conjecture:Classes_of_groups:CAT(0)-groups}
Finite dimensional $\CAT(0)$-groups;

\item\label{the:status_of_the_Full_Farrell-Jones_Conjecture:Classes_of_groups:solvable}
Virtually solvable groups;

\item\label{the:status_of_the_Full_Farrell-Jones_Conjecture:Classes_of_groups:lattices}
  (Not necessarily cocompact) lattices in path connected second countable locally compact
  Hausdorff groups.
  
  More generally, if $L$ is a (not necessarily cocompact) lattice in a second countable
  locally compact Hausdorff group $G$ such that $\pi_0(G)$ is discrete and belongs to
  $\calfj$, then $L$ belongs to $\calfj$;

\item\label{the:status_of_the_Full_Farrell-Jones_Conjecture:Classes_of_groups:pi_low_dimensional}
  Fundamental groups of (not necessarily compact) connected manifolds (possibly with
  boundary) of dimension $\le 3$;

\item\label{the:status_of_the_Full_Farrell-Jones_Conjecture:Classes_of_groups:GL}
  The groups $\GL_n(\IQ)$ and   $\GL_n(F(t))$ for $F(t)$ the function field over a finite field $F$;

\item\label{the:status_of_the_Full_Farrell-Jones_Conjecture:Classes_of_groups:S-arthmetic}
$S$-arithmetic groups;

\item\label{the:status_of_the_Full_Farrell-Jones_Conjecture:Classes_of_groups:mapping_class_groups}
  The mapping class group $\Gamma_{g,r}^s$  of a closed orientable
  surface of genus $g$ with $r$ boundary components and $s$ punctures
  for $g,r,s \ge 0$;

\item\label{the:status_of_the_Full_Farrell-Jones_Conjecture:Classes_of_groups:pi_1(graphs)}
Fundamental groups of  graphs of abelian groups;

\item\label{the:status_of_the_Full_Farrell-Jones_Conjecture:Classes_of_groups:pi_1(graphs_virtually_cyclic)}
Fundamental groups of  graphs of virtually cyclic groups;

\item\label{the:status_of_the_Full_Farrell-Jones_Conjecture:groups_in_calbc:Artin_braid_groups}
Artin's full braid groups $B_n$;

\item\label{the:status_of_the_Full_Farrell-Jones_Conjecture:groups_in_calbc:Coxeter_groups}
  Coxeter groups;

\item\label{the:status_of_the_Full_Farrell-Jones_Conjecture:finite_product_of_hyperbolic_graphs}
Groups  which acts properly and
  cocompactly on a finite product of hyperbolic graphs.

\end{enumerate}

\item\label{the:status_of_the_Full_Farrell-Jones_Conjecture:inheritance}
The class $\calfj$ has the following inheritance properties:

\begin{enumerate}

\item \emph{Passing to subgroups}%
\label{the:status_of_the_Full_Farrell-Jones_Conjecture:inheritance:passing_to_subgroups}\\
Let $H \subseteq G$ be an inclusion of groups. 
If $G$ belongs to $\calfj$, then $H$ belongs to $\calfj$;

\item \emph{Passing to finite direct products}%
\label{the:status_of_the_Full_Farrell-Jones_Conjecture:inheritance:Passing_to_finite_direct_products}\\
If the groups $G_0$ and $G_1$ belong to $\calfj$, then also $G_0 \times G_1$ belongs to $\calfj$;

\item \emph{Group extensions}%
\label{the:status_of_the_Full_Farrell-Jones_Conjecture:inheritance:group_extensions}\\
Let $ 1 \to K \to G \to Q \to 1$ be an extension of groups. Suppose that for any infinite cyclic subgroup
$C \subseteq Q$ the group $p^{-1}(C)$ belongs to $\calfj$ 
and that the groups $K$ and $Q$ belong to $\calfj$.

Then $G$ belongs to $\calfj$;

\item \emph{Group extensions with virtually torsionfree hyperbolic groups as kernel}%
  \label{the:status_of_the_Full_Farrell-Jones_Conjecture:inheritance:hyperbolic_by_infinite_cyclic}\\
  Let $ 1 \to K \to G \to Q \to 1$ be an extension of groups such that $K$ is virtually torsionfree hyperbolic
  and $Q$ belongs to $\calfj$. Then $G$ belongs to $\calfj$;

  \item \emph{Group extensions with countable free groups as kernel}%
    \label{the:status_of_the_Full_Farrell-Jones_Conjecture:inheritance:countable_free_by_infinite_cyclic}\\
    Let $ 1 \to K \to G \to Q \to 1$ be an extension of groups such that $K$ is a
    countable free group (of possibly infinite rank) and $Q$ belongs to $\calfj$. Then $G$
    belongs to $\calfj$;

\item \emph{Colimits over directed systems}%
\label{the:status_of_the_Full_Farrell-Jones_Conjecture:inheritance:directed_colimits}\\
Let $\{G_i \mid i \in I\}$ be a direct system of groups indexed by the directed set $I$
(with arbitrary structure maps). Suppose that for each $i \in I$ the group $G_i$ belongs to $\calfj$.

Then the colimit $\colim_{i \in I} G_i$ belongs to $\calfj$;

\item \emph{Passing to free products}%
\label{the:status_of_the_Full_Farrell-Jones_Conjecture:inheritance:Passing_to_free_products}\\
Consider a collection of groups $\{G_i \mid i \in I\}$ such that $G_i$ belongs $\calfj$ for each $i \in I$.
Then $\ast_{i \in I} G_i$ belongs to $\calfj$;

\item \emph{Passing to overgroups of finite index}%
\label{the:status_of_the_Full_Farrell-Jones_Conjecture:Passing_to_over_groups_of_finite_index}\\
Let $G$ be an overgroup of $H$ with finite index $[G:H]$.
If $H$ belongs to $\calfj$, then $G$ belongs to $\calfj$;

\item \emph{Graph products}%
\label{the:status_of_the_Full_Farrell-Jones_Conjecture:graph_products}\\
A graph product of groups, each of which belongs to $\calfj$, belongs to $\calfj$ again.

\end{enumerate}

\end{enumerate}

\end{theorem}

A discussion which classes of groups belong to $\calfj$ is given  in~\cite[Section~16.8]{Lueck(2022book)}, where also references to the
literature can be found. At least
we mention the following small selection of important
papers~\cite{Bartels-Bestvina(2019),Bartels-Echterhoff-Lueck(2008colim),  Bartels-Farrell-Lueck(2014), Bartels-Lueck(2012annals),
  Bartels-Lueck-Reich(2008hyper), Bartels-Lueck-Reich-Rueping(2014), Bunke-Kasprowski-Winges(2021),
  Farrell-Jones(1986a), Farrell-Jones(1987), Farrell-Jones(1993a),  Kammeyer-Lueck-Rueping(2016), Wegner(2012), Wegner(2015)}.


\typeout{---------- Section 9:  Further applications of the Full Farrell-Jones Conjecture    ---------------}

\section{Further applications of the Full Farrell-Jones Conjecture}%
\label{sec:Further_applications_of_the_Full_Farrell-Jones_Conjecture}

The are further prominent applications of the Full Farrell-Jones Conjecture~\ref{con:The_Full_Farrell-Jones_Conjecture}
which we do not treat in this article.  Key words are:

\begin{itemize}

\item The Bass Conjectures;

\item The Novikov Conjecture;

\item Non-existence of positive scalar curvature metrics on aspherical closed manifolds;

\item Whitehead spaces and pseudoisotopy spaces;

\item Fibering manifolds;

\item Homotopy invariance of the Hirzebruch-type invariant $\tau^{(2)}(M)$;

\item Homotopy invariance of the (twisted) $L^2$-torsion;

\item Classification of certain classes of manifolds.

\item Vanishing of $\kappa$-classes for aspherical closed manifolds.

\end{itemize}

More information including references to the literature about these topics can be found
in~\cite[Section~13.12]{Lueck(2022book)}.


\typeout{---------- Section 10:  Methods and strategies of proofs    ---------------}

\section{Methods and strategies of proofs}%
\label{sec:Methods_an_strategies_of_proofs}

We very briefly give some informations about the methods and strategies of the proofs of
the Full Farrell-Jones Conjecture~\ref{con:The_Full_Farrell-Jones_Conjecture}.  The key
words are \emph{assembly maps}, \emph{controlled topology}, \emph{flow spaces}, and
\emph{transfers}. More details of the strategies of proofs can be found in the survey
article by Bartels~\cite{Bartels(2016)} or in~\cite[Chapter~19]{Lueck(2022book)}.

It is rather amazing that the proof of the Full Farrell-Jones Conjecture, whose statement
is of homotopy theoretic nature,  uses input from geometric group theory and
dynamical systems. This becomes more puzzling if one observes that the Idempotent
Conjecture~\ref{con:Idempotent_Conjecture} is a purely ring theoretic statement and one
would expect that a proof of it uses only algebraic methods, whereas it turns out that its
proof for prominent classes of groups such a hyperbolic groups is carried out by proving
the $K$-theoretic part of the Full Farrell-Jones
Conjecture~\ref{con:The_Full_Farrell-Jones_Conjecture} and not by internal algebraic methods.


\subsection{Assembly maps}\label{subsec:asembly_maps}

Probably the basic idea of assembly maps goes back to Quinn~\cite{Quinn(1971),Quinn(1995a)} and
Loday~\cite{Loday(1976)} leading for instance to the $K$-theoretic assembly map
$H_n(BG;\bfK_R) \to K_n(RG)$. The $L$-theory version and a more algebraic version
was introduced by Ranicki~\cite{Ranicki(1992)}. 
The basic and uniform approach to assembly as it appears in Section~\ref{sec:The_Full_Farrell-Jones_Conjecture}
is sometimes called the Davis-L\"uck approach and was developed in~\cite{Davis-Lueck(1998)}.

Consider a covariant functor $\bfE \colon \Or(G) \to \SPECTRA$, where $\Or (G)$ is the
orbit category.  It yields a $G$-homology theory $\calh_*(-;\bfE)$ satisfying
$H_n(G/H;\bfE) = \pi_n(\bfE(G/H))$ for every subgroup $H \subseteq G$ and $n \in \IZ$.
The projection $\EGF{G}{\calf} \to G/G$ induces the assembly map
\begin{equation}
H_n(\EGF{G}{\calf};\bfE) \to H_n(G/G; \bfE) = \pi_n(\bfE(G/G)),
  \label{assembly_map_for_bfE}
\end{equation}
and the various assembly maps appearing in
Section~\ref{sec:The_Full_Farrell-Jones_Conjecture} come from specific choices for $\bfE$.

Let $\OrGF{G}{\calf}$ be $\calf$-restricted orbit category, i.e., the full subcategory of
$\Or(G)$ consisting of objects $G/H$ with $H \in \calf$. Then the assembly
map~\eqref{assembly_map_for_bfE} can be identified with the map
\begin{equation}
\pi_n(\bfp) \colon \pi_n\left(\hocolim_{\OrGF{G}{\calf}} \bfE \right) \to \pi_n(\bfE(G/G))
\label{assembly_homotopy_colimit}
\end{equation}
where the map of spectra 
\[
\bfp \colon \hocolim_{\OrGF{G}{\calf}} \bfE \to \hocolim_{\Or(G)} \bfE = \bfE(G/G)
\]
comes from the inclusion of categories $\OrGF{G}{\calf} \to \Or(G)$ and the fact that
$G/G$ is a terminal object in $\Or(G)$. For more information about homotopy colimits and
the identification of the maps~\eqref{assembly_map_for_bfE} and~\eqref{assembly_homotopy_colimit}, we refer
to~\cite[Sections~3 and~5]{Davis-Lueck(1998)}. 

This interpretation is one explanation for the name \emph{assembly}. If the assembly
map~\eqref{assembly_homotopy_colimit} is bijective for all $n \in \IZ$, or, equivalently,
the map $\bfp$ above is  a weak homotopy equivalence, we have a recipe to assemble
$\bfE(G/G)$ from its values $\bfE(G/H)$, where $H$ runs through $\calf$.  The idea is that
$\calf$ consists of well-understood subgroups, for which one knows the values $\bfE(G/H)$
for $H \subseteq G$ and hence $\hocolim_{\OrGF{G}{\calf}} \bfE$, whereas $\bfE(G/G)$ is
the object, which one wants to understand and is very hard to access.

One can describe the assembly maps also by a universal property. Roughly speaking
it is the best approximation of the functor $G/H \to \pi_n(\bfE(G/H))$ by a $G$-homology theory from the left,
see~\cite[Section~6]{Davis-Lueck(1998)}.

More information about assembly can be found for instance
in~\cite{Lueck(2019handbook)} and~\cite[Chapter~18]{Lueck(2022book)}.


\subsection{Controlled topology}\label{subsec:Controlled_topology}

Let $M_0$ and $N$ be closed topological manifold of dimension $n \ge 5$ equipped with a metric
  generating the given topology.  Then there exists $\epsilon > 0$ with the following
  properties:

  \begin{itemize}
  \item Let $M$ be a closed manifold and $f \colon M \to N$ be a homotopy equivalence which is
    \emph{$\epsilon$-controlled} in the sense  that there exists a map
    $g \colon  N \to M$ and homotopies $h \colon g \circ f \simeq \id_M$ and $k \colon f \circ g \simeq \id_N$
    such that for every $x \in M$ the diameter of $f(h(\{x\} \times I))$ and for every $y \in N$ the diameter of
    $k(\{y\}  \times I)$ is bounded by $\epsilon$.

    Then by the $\alpha$-Approximation
  due to Chapman and Ferry~\cite{Chapman-Ferry(1979)} $f$ is homotopic to a homeomorphism and in
  particular has trivial Whitehead torsion. So in order to prove that $N$ is topological
  rigid in the sense of Definition~\ref{def:topological_rigid}, it suffices to show that a
  given homotopy equivalence $g \colon M \to N$ is homotopic to an $\epsilon$-controlled
  homotopy equivalence.  Roughly speaking, to achieve up to homotopy a homeomorphism, it
  suffices to gain $\epsilon$-control.

  \item An $h$-cobordism $(W;M_0,M_1,f_0,f_1)$ over
    $M_0$ is trivial and hence has vanishing Whitehead torsion if we can show that it
    is $\epsilon$-controlled. This follows from the Thin $h$-Cobordism Theorem of
     Quinn~\cite[Theorem~2.7]{Quinn(1979a)}.
    In particular in order to show that $\Wh(\pi_1(N))$
    vanishes, it suffices to show because of  the $s$-Cobordism
    Theorem~\ref{the:s-Cobordism_Theorem} that, for any $h$-cobordism
    $(W;M_0,M_1,f_0,f_1)$ over $M_0$,  there is an $h$-cobordism
    $(W';M_0,M'_1,f'_0,f'_1)$ over $M_0$ such that $(W;M_0,M_1,f_0,f_1)$ and
    $(W';M_0,M'_1,f'_0,f'_1)$ have the same Whitehead torsion and 
    $(W';M_0,M'_1,f'_0,f'_1)$ is $\epsilon$-controlled;
  \end{itemize}

  Hence to prove for a finitely presented torsionfree group 
  Conjecture~\ref{con:FHC_Farrell-Jones_Conjecture_for_Wh(G)_for_torsionfree_Groups}
   which predicts the vanishing of $\Wh(G)$, or the Borel
  Conjecture~\ref{con:Borel_Conjecture}, which predicts the topological rigidity
  of an aspherical closed manifold with fundamental group $G$, a promising strategy is to
  gain control, i.e., turning an $h$-cobordism or a homotopy equivalence to an
  $\epsilon$-controlled one without changing the  class associated to it in the Whitehead group.

  This turns out to be a major breakthrough since it allows to bring in completely new methods, namely,
  geometric methods, into the play.  Note  that the  pioneering work introducing control ideas into assembly
  was done by Farrell and Hsiang in several papers, including~\cite{Farrell-Hsiang(1970),Farrell-Hsiang(1978b)}.
  These ideas were then extensively developed by Farrell and Jones,  in
  particular in their seminal papers~\cite{Farrell-Jones(1986a),Farrell-Jones(1987)}.

  \begin{remark}[Control-Strategy]\label{rem:Control-Strategy}
  The considerations above lead to the  following
  \emph{Control-Strategy}\index{strategy!Control-Strategy}\index{control!Control-Strategy}
  for proving the Farrell-Jones Conjecture:

  \begin{enumerate}
  \item\label{Control-Strategy:Interpretation_of_the_source} Interprete elements in the
    target group $K_n(\IZ G)$ of the assembly map as a kind of cycles and the source of
    the assembly map $H_n^G(\EGF{G}{\calvcyc};\bfK_R)$ as \emph{controlled cycles}, i.e.,
    cycles satisfying certain control conditions related to the family $\calvcyc$;
    
  \item\label{Control-Strategy:Interpretation_of_the_assembly_map} Identify the assembly
    map as a kind of \emph{forget control map};

  \item\label{Control-Strategy::surjectivity} For a specific group $G$ and the 
    family $\calvcyc$, develop a strategy how to change a cycle without changing its class
    in $K_n(\IZ G)$ such that the new representative satisfies the necessary control
    conditions to ensure that the it defines an element in
    $H_n^G(\EGF{G}{\calvcyc};\bfK_R)$. This proves surjectivity of the assembly map. One may
    call this process \emph{gaining control};\index{control!gaining
      control}\index{Strategy!!gaining control}

  \item\label{Control-Strategy::injectivity} Use a relative version of
    part~\ref{Control-Strategy::surjectivity} to prove injectivity of the assembly
    map. One may call this process \emph{gaining relative control}.
  \end{enumerate}
  
  The strategy for $L$-theory is completely analogous.

\end{remark}

\begin{example}[Singular homology]\label{exa:singular_homology}
  Next we illustrate this strategy in a much easier and classical instance, namely,
  singular homology, by repeating how one proves excision for it.

  Let $X$ be a topological space, and let
  $C_*^{\operatorname{sing}}(X;R)$ be the singular
  chain complex of $X$ with coefficients in the ring $R$.  Let
  $\calu = \{U_i \mid i \in I\}$ be a cover of $X$, i.e., a collection of subsets $U_i$
  such that the union of their interiors $U_i^{\circ}$ is $X$. Denote by $S_n^{\calu}(X)$
  the subset of the set $S_n(X)$ of those singular $n$-simplices
  $\sigma \colon \Delta_n \to X$ for which there exists $i \in I$ satisfying
  $\im(\sigma) \subseteq U_i$.  Let $C_*^{\operatorname{sing},\calu}(X;R)$ be the $R$
  subchain complex of $C_*^{\operatorname{sing}}(X;R)$ whose $n$th chain module consists
  of elements of the shape $\sum_{\sigma \in S^{\calu}_n(X)} r_{\sigma} \cdot \sigma$. Let
  $i _*^{\calu} \colon C_*^{\operatorname{sing},\calu}(X;R) \to
  C_*^{\operatorname{sing}}(X;R)$ be the inclusion. The main ingredient in the proof of
  excision is to show that $i_*$ is a homology equivalence.  Then excision follows by
  applying the result above to $\calu = \{X \setminus A, B\}$ for
  $A \subseteq B \subseteq X$ with $\overline{A} \subseteq B^{\circ}$.

  The proof that
  $i _*^{\calu} \colon C_*^{\operatorname{sing},\calu}(X;R) \to
  C_*^{\operatorname{sing}}(X;R)$ is a homology equivalence is based on the construction
  of the subdivision operator which subdivides $\Delta_n$ into a bunch of smaller copies
  of $\Delta_n$ and replaces the singular simplex $\sigma \colon \Delta_n \to X$ by the
  sum of the singular simplices obtained by restricting to these smaller pieces.  This
  process does not change the homology class but can be used to arrange that the
  representing cycle lies in $C_*^{\operatorname{sing},\calu}(X;R)$. This implies
  surjectivity of
  $H_n(i _*^{\calu}) \colon H_n(C_*^{\operatorname{sing},\calu}(X;R)) \to
  H_n(C_*^{\operatorname{sing}}(X;R))$.  One obtains injectivity by applying these
  construction to an $(n+1)$-simplex $\tau \colon \Delta_{n+1} \to X$, provided that the
  restriction of $\tau$ to faces of $\Delta_{n+1}$ does already lie in $S_n^{\calu}(X)$.

  Roughly speaking, the process of gaining control is realized by subdivision.
\end{example}

For more information about controlled topology we refer for instance to~\cite[Section~19.4
and Chapter~21]{Lueck(2022book)}.


\subsection{Flow spaces}\label{subsec:Flow_spaces}

To gain control, flow spaces come into play. Roughly speaking, one considers so called
\emph{geometric modules} over a metric space $X$ which is essentially a collection of
finitely generated free $R$-modules $\{M_x \mid x \in X\}$ which is locally finite, i.e.,
for $y \in X$ there exists an open neighborhood $U$ of $y$ in $X$ such that its support
$\supp(M) = \{x \in X \mid M_x \not= \{0\}\} \cap U$ is finite.  Very roughly speaking,
if $X$ comes with an appropriate flow,
one can apply the flow to arrange that the support has arbitrary
small diameter.  Thus one gains control.  Since the flow is continuous this does not change
the class represented in the $K$- or $L$-theory by such a geometric module (or quadratic form over it)
or by an
automorphisms of such a geometric module (or quadratic form over it).

The appropriate flow spaces come from the geometry of the group $G$ under consideration.
Suppose that $G$ is the fundamental group of a hyperbolic closed manifold $M$ of dimension
$d$.  Then we take $X = \widetilde{M}$ which is just the hyperbolic space $\IH^d$.  The
basic idea is that for two points $x_0$ and $x_1$ in $\IH^d$ and a point $z$ on the
boundary $\partial \IH^d$ there exists $\tau \in \IR$ such that
$\lim_{t \to \infty} d(\Phi^z_{t}(x_0), \Phi^z_{t + \tau}(x_1)) = 0$, if $\Phi^z_{t}(x_0)$
denotes the point obtained by moving $x_0$ to $z$ along the geodesic ray starting at $x_0$
to $z$ for the time $t$.

For a survey on the construction of the appropriate flow spaces for certain classes of groups
we refer for instance to~\cite[Chapter~22]{Lueck(2022book)}.


\subsection{Transfers}\label{subsec:Transfers}

We at least describe how transfer maps occurred in the work Farrell-Jones.  Consider a
closed hyperbolic manifold $M$ of dimension $d \ge 5$.  We want to show
$\Wh(\pi_1(M)) = 0$.  There is no flow on $M$ but there is the geodesic flow on the sphere
tangent bundle $STM$.  There is the  transfer map
\[
  p^* \colon \Wh(\pi_1(M)) \to \Wh(\pi_1(STM)).
\]
It sends the Whitehead torsion of an
$h$-cobordism $(W;M,M',f,f')$ over $M$ to the Whitehead torsion of the $h$-cobordism over
$STM$ obtain by the pullback construction associated to the projection
$p \colon STM \to M$ and an appropriate retraction $r \colon W \to M$.  The flow on $STM$
allows to show that $\Wh(\pi_1(STM))$ vanishes. The projection $p$ induces an isomorphisms
$p_* \colon \Wh(\pi_1(STM)) \to \Wh(\pi_1(M))$ such that
$p_* \circ p^* \colon \Wh(\pi_1(M)) \to \Wh(\pi_1(M))$ is multiplication with the Euler
characteristic of the fiber of $p$ which is $1 - (-1)^{d}$. If $d$ is odd, we conclude
that $p_* \circ p^* = 2 \cdot \id_{\Wh(\pi_1(M))}$ and hence each element in $\Wh(\pi_1(M))$
  has order at most $2$. Farrell-Jones replace $STM$ by a  certain semisphere bundle
  $q \colon S^+TM \to M$. Its fiber is the upper hemisphere which is contractible and has
  therefore Euler characteristic $1$.  This implies $q_* \circ q^* =
  \id_{\Wh(\pi_1(M))}$. Since one can still show that $\Wh(\pi_1(S^+TM))$ vanishes using
    the geodesic flow and $q_* \colon \Wh(\pi_1(S^+TM)) \to \Wh(\pi_1(M))$ is bijective, we
    get $\Wh(\pi_1(M)) = \{0\}$.

    For more information about transfer methods we refer for instance to~\cite[Chapter~23]{Lueck(2022book)}.


    \subsection{The role of the family $\calvcyc$}\label{subsec:The_role_of_the_family_calvcyc}

    Consider an appropriate flow space $X$ with an isometric proper $G$-action. Let
    $c \colon \IR \to X$  be a geodesic.  Then the subgroup
    $G_c = \{g \in G \mid g \cdot \im(c) = \im(c)\}$ is an extension of the shape
    $1 \to \IZ \to G_c \to G_{c(0)} \to 1$ where the finite group $G_{c(0)}$ is the
  isotropy group of $c(0)$ in $X$. Groups which contain a normal infinite cyclic subgroup
  of finite index are called virtually cyclic of type I. If $\calvcyc_I$ is the class of
  virtually cyclic subgroups of type I of $G$, then one gets the bijectivity of the
  assembly map
  \[
H_n^G(\pr;\bfK_{\cala})  \colon H^G_n(\EGF{G}{\calvcyc_I};\bfK_{\cala}) 
\to H_n^G(G/G;\bfK_{\cala}) =  K_n(\cala[G]),
\]
if the strategy of proof described above works out. This is consistent with
Conjecture~\ref{con:K-theoretic_Farrell-Jones_Conjecture_with_coefficients_in_additive_G-categories},
since one can show for every group $G$, every additive $G$-category $\cala$, and every
$n \in \IZ$ that the relative assembly map
\[
H_n^G(\pr;\bfK_{\cala})  \colon H^G_n(\EGF{G}{\calvcyc_I};\bfK_{\cala}) 
\to H_n^G(\EGF{G}{\calvcyc};\bfK_{\cala})
\]
is bijective.

This reduction from $\calvcyc$ to $\calvcyc_I$ does not holds for $L$-theory, since in the
proof for $L$-theory one has to enlarge the family $\calvcyc_I$ to the family
$\{H \subseteq G \mid \exists H' \subseteq H \; \text{with}\; [H:H'] \le 2, H \in
\calvcyc_I\}$ which is precisely $\calvcyc$. A general discussion when one can make the
family $\calvcyc$ smaller, can be found for instance
in~\cite[Section~13.10]{Lueck(2022book)}.

The family $\calvcyc$ seems to be the smallest family for which
the Full Farrell-Jones Conjecture~\ref{con:The_Full_Farrell-Jones_Conjecture}
might be true in general.


\typeout{---------- Section 11:  Can the Full Farrell-Jones Conjecture be true for all groups?    ---------------}

\section{Can the Full Farrell-Jones Conjecture be true for all groups?}%
\label{sec:Can_the_Full_Farrell-Jones_Conjecture_be_true_for_all_groups?}

We are not aware of any group for which the Full Farrell-Jones
Conjecture~\ref{con:The_Full_Farrell-Jones_Conjecture} is known to be false.

Here is a list of interesting groups for which  the Full Farrell-Jones Conjecture~\ref{con:The_Full_Farrell-Jones_Conjecture}
is open in general at the time of writing:

\begin{itemize}

\item elementary amenable, amenable, or a-T-menable groups;
\item $\Out(F_n)$ for $n \ge 3$;
\item Artin groups; 
\item Thompson's groups $F$, $V$, and $T$;
\item Torsionfree one-relator groups;
\item Linear groups;
\item Subgroups of almost connected Lie groups;
\item Residually finite groups;
\item (Bi-)Automatic groups;
\item Locally indicable groups.
\end{itemize}

It is hard to believe that the Full Farrell-Jones
Conjecture~\ref{con:The_Full_Farrell-Jones_Conjecture} is true for all groups, since there
have been so many prominent conjectures about groups which were open for some time and
for which finally counterexamples were found. On the other hand, the conjecture is known
for so many groups so that we currently have no strategy to find counterexamples.
There are many exotic properties of groups, where one may hope  that they rule out the
Full Farrell-Jones Conjecture~\ref{con:The_Full_Farrell-Jones_Conjecture}, e.g.,
groups with expanders, but for all these properties one knows  examples of groups which possess
this property and which do satisfy the Full Farrell-Jones
Conjecture~\ref{con:The_Full_Farrell-Jones_Conjecture}. One reason is
that often constructions of group with exotics properties are carried out as colimits of hyperbolic groups but
these do satisfy the Full Farrell-Jones
Conjecture~\ref{con:The_Full_Farrell-Jones_Conjecture} because of
Theorem~\ref{the:status_of_the_Full_Farrell-Jones_Conjecture}~%
\ref{the:status_of_the_Full_Farrell-Jones_Conjecture:Classes_of_groups:hyperbolic_groups}
and~\ref{the:status_of_the_Full_Farrell-Jones_Conjecture:inheritance:directed_colimits}.

Actually, there is a non-trivial result which holds for all groups and is a consequence of the 
Full Farrell-Jones
Conjecture~\ref{con:The_Full_Farrell-Jones_Conjecture}.
Let $i\colon H \to G$ be the inclusion of a normal subgroup
$H \subset G$.  It induces a homomorphism $i_0\colon \Wh(H) \to \Wh(G)$.  The conjugation
actions of $G$ on $H$ and on $G$ induce $G$-actions on $\Wh(H)$ and on $\Wh(G)$ which
turns out to be trivial on $\Wh(G)$. Hence $i_0$ induces homomorphisms
\begin{eqnarray}
i_1\colon  \IZ \otimes_{\IZ G} \Wh(H) & \to & \Wh(G);
\label{Wh(H)_to_Wh(G):i_1}
\\
i_2\colon  \Wh(H)^G  & \to & \Wh(G).
\label{Wh(H)_to_Wh(G):i_2}
\end{eqnarray}

\begin{theorem}\label{the:Wh(H)_to_Wh(G)}
Let $i\colon  H \to G$ be the inclusion of a normal finite subgroup $H$
into an arbitrary group $G$. Then the maps
$i_1$ and $i_2$ defined in~\eqref{Wh(H)_to_Wh(G):i_1} and~\eqref{Wh(H)_to_Wh(G):i_2}
have finite kernels.
\end{theorem}
\begin{proof}
See~\cite[Theorem~9.38 on page~354]{Lueck(2002)}.
\end{proof}

Furthermore, Yu~\cite[Theorem~1.1]{Yu(2017)}, see also
Cortinas-Tartaglia~\cite{Cortinas-Tartaglia(2014)}, proved  that the
$K$-theoretic assembly map
\[
  H_n^G(\EGF{G}{\calvcyc};\bfK_{\cals}) \to K_n(\cals G)
\]
is rationally injective for every group $G$, where $\cals$ is the ring
of Schatten class operators of an infinite dimensional separable Hilbert space.

\begin{remark}[Universal finitely presented groups]\label{rem:Universal_finitely_presented_groups}
  The following observation is amusing  that
  the  Full Farrell-Jones
  Conjecture~\ref{con:The_Full_Farrell-Jones_Conjecture} holds for all groups
  if it holds for one specific finitely presented group $U$.
  
  Namely, let $U$ be a group that  is universal finitely presented, i.e., any
  finitely presented group is isomorphic to a subgroup of $G$.  Such a group exists by
  Higman~\cite[page~456]{Higman(1961)}, and there is even a universal finitely presented
  group which is the fundamental group of a
  complement of an embedded $S^3$ in $S^5$,
  see~\cite[Corollary~3.4]{Gonzalez-Acuna-Gordon-Simon(2010)}.

  Suppose  that $U$ satisfies the  Full Farrell-Jones
  Conjecture~\ref{con:The_Full_Farrell-Jones_Conjecture}. Let $G$ be any  group.
It is the directed  union of its finitely generated subgroups. 
Hence by Theorem~\ref{the:status_of_the_Full_Farrell-Jones_Conjecture}~%
\eqref{the:status_of_the_Full_Farrell-Jones_Conjecture:inheritance:directed_colimits}
the Full Farrell-Jones Conjecture~\ref{con:The_Full_Farrell-Jones_Conjecture}
holds for all groups if and only if it holds for all finitely generated groups.
Any finitely generated group can be written  as a directed colimit of finitely presented groups. 
Hence by Theorem~\ref{the:status_of_the_Full_Farrell-Jones_Conjecture}~%
\eqref{the:status_of_the_Full_Farrell-Jones_Conjecture:inheritance:directed_colimits}
the Full Farrell-Jones Conjecture~\ref{con:The_Full_Farrell-Jones_Conjecture}
holds for all finitely generated groups if and only if holds for all finitely presented  groups.
Now the claim follow from Theorem~\ref{the:status_of_the_Full_Farrell-Jones_Conjecture}~%
\eqref{the:status_of_the_Full_Farrell-Jones_Conjecture:inheritance:passing_to_subgroups},
since every finitely presented group is a subgroup of $U$.
\end{remark}


\typeout{----------------------------- References ------------------------------}

\addcontentsline{toc<<}{section}{References} 


\end{document}